\documentclass[a4paper, 10pt]{preprint}

\usepackage[left=1.5in, right=1.5in, bottom=1.5in]{geometry}
\usepackage{comment}
\usepackage{mhenvs}
\usepackage{mhsymb}
\usepackage{orcidlink}

\usepackage[full]{textcomp}
\usepackage[osf]{newtxtext}
\usepackage{amssymb}
\usepackage{amsmath}
\usepackage{mathtools}
\usepackage{mhequ}
\usepackage{booktabs}
\usepackage{tikz}
\usepackage{mathrsfs}
\usepackage{longtable}
\usepackage{microtype}
\usepackage{wasysym}
\usepackage{centernot}
\usepackage{enumitem}

\usepackage{hyperref}

\usepackage[backend=biber,style=alphabetic]{biblatex}
\usepackage{crossreftools}
\usepackage{physics}
\usepackage[bbgreekl]{mathbbol}
\usepackage{pgfplots}
\pgfplotsset{width=6.7cm,compat=1.9}

\def\loc{\mathrm {loc}}
\def\pvar{p\mathrm{\text{-}var}}

\mathtoolsset{showonlyrefs=true}

\newif\ifdark
\IfFileExists{dark}{\darktrue}{\darkfalse}

\ifdark

\definecolor{darkred}{rgb}{0.9,0.2,0.2}
\definecolor{darkblue}{rgb}{0.7,0.3,1}
\definecolor{darkgreen}{rgb}{0.1,0.9,0.1}
\definecolor{pagebackground}{rgb}{.15,.21,.18}
\definecolor{pageforeground}{rgb}{.84,.84,.85}
\pagecolor{pagebackground}
\AtBeginDocument{\globalcolor{pageforeground}}

\else

\definecolor{darkred}{rgb}{0.7,0.1,0.1}
\definecolor{darkblue}{rgb}{0.4,0.1,0.8}
\definecolor{darkgreen}{rgb}{0.1,0.7,0.1}
\definecolor{pagebackground}{rgb}{1,1,1}
\definecolor{pageforeground}{rgb}{0,0,0}

\fi

\definecolor{newred}{RGB}{208,16,76}
\definecolor{newgreen}{RGB}{34,125,81}

\makeatletter
\newcommand{\globalcolor}[1]{%
	\color{#1}\global\let\default@color\current@color
}
\makeatother

\DeclareSymbolFont{timesoperators}{T1}{ptm}{m}{n}
\SetSymbolFont{timesoperators}{bold}{T1}{ptm}{b}{n}

\makeatletter
\renewcommand{\operator@font}{\mathgroup\symtimesoperators}
\makeatother

\DeclareMathAlphabet{\mathbbm}{U}{bbm}{m}{n}
\DeclareFontFamily{U}{BOONDOX-calo}{\skewchar\font=45 }
\DeclareFontShape{U}{BOONDOX-calo}{m}{n}{
	<-> s*[1.05] BOONDOX-r-calo}{}
\DeclareFontShape{U}{BOONDOX-calo}{b}{n}{
	<-> s*[1.05] BOONDOX-b-calo}{}
\DeclareMathAlphabet{\mcb}{U}{BOONDOX-calo}{m}{n}
\SetMathAlphabet{\mcb}{bold}{U}{BOONDOX-calo}{b}{n}

\makeatletter 

\newcommand*{\fat}{}
\DeclareRobustCommand*{\fat}{%
	\mathbin{\mathpalette\bigcdot@{}}}
\newcommand*{\bigcdot@scalefactor}{.5}
\newcommand*{\bigcdot@widthfactor}{1.15}
\newcommand*{\bigcdot@}[2]{%
	\sbox0{$#1\vcenter{}$}
	\sbox2{$#1\cdot\m@th$}%
	\hbox to \bigcdot@widthfactor\wd2{%
		\hfil
		\raise\ht0\hbox{%
			\scalebox{\bigcdot@scalefactor}{%
				\lower\ht0\hbox{$#1\bullet\m@th$}%
			}%
		}%
		\hfil
	}%
}

\DeclareRobustCommand{\TitleEquation}[2]{\texorpdfstring{\StrLeft{\f@series}{1}[\@firstchar]$\if b\@firstchar\boldsymbol{#1}\else#1\fi$}{#2}}

\makeatother

\theoremnumbering{Alph}
\renewtheorem{theorem*}{Theorem}
\renewtheorem{corollary*}{Corollary}
\renewtheorem{proposition*}[theorem*]{Proposition}
\newtheorem{assumption}[lemma]{Assumption}
\newtheorem{example}[lemma]{Example}
\renewtheorem*{example*}{Example}
\numberwithin{equation}{section}

\def\slash{\leavevmode\unskip\kern0.18em/\penalty\exhyphenpenalty\kern0.18em}
\def\dash{\leavevmode\unskip\kern0.18em--\penalty\exhyphenpenalty\kern0.18em}

\let\epsilon\varepsilon

\def\${|\!|\!|}

\setlist{noitemsep,topsep=4pt}
\def\para_#1{/\!\!/_{\!#1}}

\DeclareMathOperator{\Span}{Span}
\DeclareMathOperator{\Proj}{Proj}

\makeatletter
\newcounter{longlist}
\def\thelonglist{\@roman\c@longlist}
\def\labellonglist{(\thelonglist)}
\def\longlist{\@ifnextchar[{\@longlist}{\@longlist[]}}
\def\@longlist[#1]{%
    \list
        {\labellonglist}%
        {\labelwidth\z@ \itemindent-\leftmargin
            \usecounter{longlist}%
            \let\makelabel\descriptionlabel
                \setcounter{longlist}{3}%
                \settowidth\labelwidth{\labellonglist}%
                \setcounter{longlist}{0}%
            \itemsep0pt \topsep\medskipamount
            \parsep0pt \partopsep0pt
            }}

\makeatother

\providecommand{\burl}[1]{\url{#1}}

\def\f{\frac}

\begin{document}

\title{Strong completeness of SDEs and non-explosion for RDEs  \\  with coefficients having unbounded derivatives}
\author{Xue-Mei Li$^{1}$  \orcidlink{0000-0003-1211-0250} and Kexing Ying$^{2}$ \orcidlink{0000-0002-8292-4746}}

\institute{EPFL, Switzerland and Imperial College London, U.K.  \email{xue-mei.li@epfl.ch} \and EPFL, Switzerland.  \email{kexing.ying@epfl.ch}}

\maketitle
\begin{abstract}
  We establish a non-explosion result for rough differential equations (RDEs) in which the noise 
  and drift coefficients, together with their derivatives, may grow unboundedly at infinity. 
  In addition, we prove the existence of a global bi-continuous solution flow for stochastic differential equations (SDEs). 
  Finally, the non-explosion results for RDEs are shown to be sharp by constructing counterexamples.
\end{abstract}
\tableofcontents

\section{Introduction}

Recent advancements such as the stochastic sewing lemma \cite{Le2020} have significantly enhanced
our understanding of local well-posedness of rough differential equations (RDEs) and stochastic
differential equations (SDEs). Building on these developments, this article investigates two key
problems:
\begin{enumerate}
  \item Strong completeness for SDEs driven by Brownian motions and related processes; 
  \item Global existence for RDEs.
\end{enumerate}
Our aim is to accommodate vector fields with unbounded derivatives, while
allowing super-linear growth in the rotational component of the drift vector field
(i.e. the part orthogonal to the radial direction).

We first consider an ordinary differential equation (ODE) perturbed by a Hölder
continuous path, and then apply the findings to SDEs with additive noise, where we establish strong completeness and uniform non-explosion results. Leveraging an interactive bound, we further apply these results to controlled RDEs, deriving a criterion for strong completeness in the case of SDEs with multiplicative noise.  
A key technical ingredient is the Flying Fish Lemma (Lemma \ref{lem:poly-bound}), which enables us to select a common time interval on which the rough path norm can be effectively controlled. This, in turn, yields a priori estimates for the rough integral in the RDE setting and allows the rough integral to be treated as in the additive noise case.

We work on a filtered probability space $(\Omega, \mathcal{F}, (\mathcal{F}_t)_{t \ge 0}, \mathbb{P})$ 
satisfying the standard assumptions, and let $W_t=(W_t^1, \dots, W_t^m)$ be a
standard $(\mathcal{F}_t)$-adapted Brownian motion on $\R^m$. We consider the SDE on $\R^d$
\begin{equation}\label{eq:sde-intro}
  \dd x_t = b(x_t) \dd t + \sum_{k=1}^m\sigma_k(x_t) \dd W_t^k,
\end{equation}
where $b:\R^d\to \R^d$ is the drift vector field, and $\sigma_k: \R^d\to \R^d$ are the diffusion coefficients
modelling the stochastic influences. The associated solution flow, denoted by $\phi_t(x)$, gives
the state at time $t$ of the solution starting from the initial condition $x$.

The SDE~\eqref{eq:sde-intro} is said to be {\it complete} (or conservative) if, for any initial condition,
it admits a global in time solution almost surely. It has {\it uniform non-explosion} if, on a subset of $\Omega$ of full measure, no explosion occurs for any initial condition.
Finally, the SDE is {\it strongly complete} if its solution flow $\phi_t(x)$ admits a continuous global version, meaning that the map $(x,t)\mapsto \phi_t(x, \omega) $ is almost surely continuous. This property, described as having a continuous global solution flow, ensures stability with respect to initial data.

Strong completeness is significant for numerical simulations and is essential for constructing perfect 
cocycles in the framework of random dynamical systems. It also plays
a key role, together with the derivative flow, in establishing derivative formulas for Markov semigroups $P_t$ and in the commutation formula \cite[Section~9]{Li:94a} linking the derivative flow to non-explosion.
Specifically, for a function $f: \R^d\to \R$,
$$D(P_tf)(x_0)v_0 =\mathbb E[Df(x_t)v_t\1_{t<\xi(x)}],$$
where $Df(x)v$ is the derivative of a function $f$ at $x$ in the direction $v$, $\xi(x)$
is the explosion time of the solution starting from $x$, and $(v_t)$ solves the
linearized equation \eqref{eq:deriv-sde} (the derivative flow).

While strong completeness and completeness coincide for ODEs with $C^1$ vector fields, this
equivalence fails for SDEs. The first counterexample was given in \cite{Elworthy:78}
where the vector fields grow quadratically. A further counterexample was presented in \cite{Li:09}, in which
the coefficients of the SDE are bounded and smooth. A detailed discussion of this distinction,
together with a comparison, is provided in Section~\ref{sec:str-comp}.
\bigskip

We next introduce RDEs on a Hilbert space $H$ of the form
$$\dd x_t= b(x_t) \dd t + \sigma(x_t) \dd  \mathbb{\Gamma}_t.$$
Here, $V$ is a Banach space, $b : H \to H$ and $\sigma : H \to \CL(V; H)$
are Borel measurable, and $\mathbb{\Gamma}$ is a rough path on $V$ with regularity
$\alpha \in (\frac{1}{3}, 1)$. Here, $\CL(V; H)$ denotes the space of bounded linear operators
from $V$ to $H$. It is well known that if the vector fields and their derivatives are globally bounded
and sufficiently regular, the RDE does not explode in finite time (see for example \cite{Lyons:94, Friz:20}).

In this article, we remove the boundedness assumptions on the vector fields and their derivatives.
We establish a growth condition for non-explosion and, in the additive case, give a sharpness example in Proposition~\ref{pro:sharp}. We then apply this condition to SDEs to obtain criteria on the coefficients ensuring strong completeness.

\subsection{Main Results}

We briefly outline the main results of this article.

A non-decreasing function $f : \R_+ \to \R_+$ is called a control function if
$\int_r^\infty \frac{1}{f(s)} \dd s = \infty$ for some $r > 0$.
Throughout the article, we impose the following growth condition on the drift $b$.
\begin{assumption}\label{as:lin-growth}
  We say that $b : \mathbb{R}_+ \times H \to H$ has (at most) radial linear growth with control $f$ if,
  for all $x, t$,
  \begin{equation}\label{cond:cos}
    \<\frac{x}{\|x\|}, b(t, x)\> \le f(\|x\|).
  \end{equation}
\end{assumption}

This assumption is sharp for preventing explosion in SDEs.
Indeed, by It\^o's formula, the solution of the SDE $\dd x_t = b(x_t) \dd t + \dd W_t$ on $\R^d$ satisfies
$$\dd \|x_t\| = \<\frac{x_t}{\|x_t\|}, b(x_t)\> \dd t + \<\frac{x_t}{\|x_t\|}, \dd W_t\>.$$
Then, the comparison test together with Feller's test for explosion \cite[Theorem 5.29]{Karatzas:98} readily
shows that the SDE may explode should \eqref{cond:cos} fail.

We prove a \textit{uniform} non-explosion principle for RDEs and SDEs and establish strong completeness
for a class of equations whose drift $b$ has radial linear growth and satisfies the following growth condition
$$\|b(x)\|\lesssim f(\|x\|)^{1+ {\f 1 2} \alpha-}.$$
Unlike in the case of mere non-explosion, controlling only the growth along the outward radial direction is insufficient for uniform non-explosion (cf. \cite[Theorem 6.2]{Li:94a}, where additional derivative bounds
were imposed). This is illustrated by the orthogonally sheared Ornstein–Uhlenbeck SDE:
\begin{equation}\label{eq:SOU}
  \dd x_t = \left(- \frac{\alpha}{2}x_t + \rho(\|x_t\|)
  \begin{pmatrix}
      0 & -1 \\
      1 & 0
    \end{pmatrix}x_t\right) \dd t + \dd W_t, \qquad x_0 = x \in \R^2
\end{equation}
for some $\alpha > 0$. In this case,  $\langle x, b(x)\rangle = -\frac{\alpha}{2} \|x\|^2 \le 0$.
However, through private communication we learned that one can construct a function $\rho : \R_+ \to \R_+$  such that  $\rho'(r) = Cr^{3 + \epsilon}$ for which the above SDE is not strongly 
complete \cite{Chemnitz-Engel-Scheutzow}. This was mentioned informally in \cite{Scheutzow:17}, 
and the specific example will appear in a forthcoming publication.

We now turn to our non-explosion result, allowing the derivatives of the coefficients to grow without imposing any boundedness assumptions. 

In the two theorems below, we take $H$ to be a Hilbert space and $V$ to be a Banach space.

\begin{theorem*}[Theorem \ref{thm:rough}]\label{thm:intro-rough}
  Let $b \in \Lip_{\loc}(H; H) $, $\sigma \in C^{3}(H; \CL(V; H))$, and let $\mathbb{\Gamma}$ be an $\alpha$-H\"older rough path taking value in $V$ for some $\alpha \in (\frac{1}{3}, \frac{1}{2}]$. Suppose that $b$
  has radial linear growth with control $f$ (Assumption~\ref{as:lin-growth}), and that there exists $\kappa \in [0, \frac{1}{2})$ such that for all $x \in H$,
  \begin{itemize}
    \item $\|b(x)\| \le f(\|x\|)^{1 + \kappa \alpha}$,
    \item $\|D^n \sigma(x)\| \le f(\|x\|)^{(1 - n\kappa)\alpha-}$ for $n = 0, 1, 2$.
  \end{itemize}
  Then, the RDE $\dd x_t = b(x_t) \dd t + \sigma(x_t) \dd \mathbb{\Gamma}_t$
  is globally well posed.
\end{theorem*}
Here $D^n\sigma(x)\in\CL(\otimes^n H; \CL(V; H))$ denotes the $n$-th Fr\'echet derivative of $\sigma$ at $x$
and $\|D^n\sigma(x)\|$ is its operator norm.
We remark that, unlike \cite{Riedel:16}, we do not require $\mathbb{\Gamma}$ to be weakly geometric.
We also obtain the corresponding result for Young differential equations (YDE) (Theorem \ref{thm:young}).
\begin{theorem*}[Theorem \ref{thm:young}]
  Let $b\in \Lip_{\loc} (H; H) $, $\sigma\in C^{2}(H; \CL(V; H))$ and
  $\gamma \in C^\alpha_{\loc}(\R_+; V)$ for some $\alpha > \frac{1}{2}$.
  Then, if $b$ satisfies Assumption~\ref{as:lin-growth} with control $f$, and there exists
  constant $\kappa \in [0, 1)$ such that for all $x \in H$,
  \begin{itemize}
    \item $\|b(x)\| \le f(\|x\|)^{1 + \kappa \alpha},$
    \item $\|D^n \sigma(x)\| \le f(\|x\|)^{(1 - n\kappa) \alpha-}$ for $n = 0, 1$.
  \end{itemize}
  Then, the YDE $\dd x_t = b(x_t) \dd t + \sigma(x_t) \dd \gamma_t$
  is globally well-posed.
\end{theorem*}

We observe that Theorem~\ref{thm:intro-rough} is 
essentially sharp at least in the case where \(\kappa = 0\) and \(\alpha = \frac{1}{2}\) in the following sense.

\begin{example*}
  Let $\sigma : \R^2 \to \R^2$ be given by $\sigma(x) = \left(x_1^{\f 12 - \delta}\sin x_2, x_1^{\f 12 + \epsilon}\right)^T$ for 
  some $\epsilon > \delta > 0$ (so that \(\sigma\) and its derivative has polynomial growth of order 
  \(\f 12 + \epsilon\) and \(\f 12 - \delta\) respectively). 
  Let $\mathbb{\Gamma}$ be the pure area rough path $(0, \, (1\otimes 1) t)$. 
  Then, the RDE $\dd x(t)=\sigma(x(t))\dd\mathbb{\Gamma}_t$ reduces to the ODE 
  \begin{align*}
    \dd x(t) & = D\sigma(x(t))\sigma(x(t)) \dd \mathbb{\Gamma}_t\\
    & = 
    \begin{pmatrix}
      \left(\f 12 - \delta\right)x_1(t)^{- 2\delta} \sin^2 (x_2(t)) + x_1(t)^{1 + (\epsilon - \delta)} \cos (x_2(t))\\
      \left(\f 12 + \epsilon\right) x_1(t)^{\epsilon - \delta} \sin (x_2(t))
    \end{pmatrix}.  
  \end{align*}
  Thus, for any initial conditions of the form $(x_1, 0)$ with $x_1 > 0$, we observe that $x_2(t)=0$ 
  for all $t\ge 0$ and $\dot x_1 = x_1^{1 + (\epsilon - \delta)}$ from which we observe 
  explosion in finite time.
\end{example*}

The growth conditions in Theorem~\ref{thm:intro-rough} allow us to conclude strong completeness
of the corresponding SDE.
\begin{theorem*}[Corollary~\ref{cor:strong-complete}]
Suppose that $b$ and $ \sigma$ satisfy the conditions of Theorem~\ref{thm:intro-rough}, the SDE
  $$\dd x_t = b(x_t) \dd t + \sigma(x_t) \dd W_t$$
  is strongly complete.
\end{theorem*}

For additive noise, these conditions can be refined to allow for lower regularity in the 
drift and relaxed growth restrictions. 
\begin{proposition*}[Corollary \ref{cor:global-holder}]\label{cor:B}
  Let $\alpha \in (0, 1)$ and $\gamma\in C_{\loc}^\alpha$ be fixed. Let
  $b: [0,T]\times H\to H$ satisfies Assumption \ref{as:lin-growth} (radial linear growth with control~$f$), 
  and suppose that there exists some $\beta \in (1, 1 + \alpha]$ and $a>0$ such that
  \begin{align}
    & \sup_{t \in [0, T]}   \left|\left\langle y, b(t, x)\right\rangle\right| \le (1 + \|x\|)f(\|x\|)^\beta, 
        \quad  \hbox{ for all } y \perp  x, \|y\|=1,\label{eq:orthogonal-growth}\\
    & \sup_{t \in [0,T]}\sup_{\|x\|\le a} \|b(t,x)\|<\infty. \label{eq:boundedness}
  \end{align}
Then any maximal solution $(x_t)_{t \in [0, \xi)}$ of the perturbed ODE 
$\dd x_t = b(t, x_t)\dd t + \dd \gamma_t$ is in fact global.
\end{proposition*}

{\it Remark.}  For \(\alpha = 1\),  a simple Gr\"onwall argument shows that the conclusion above  
still holds without the orthogonal growth condition \eqref{eq:orthogonal-growth}.

Corollary~\ref{cor:global-holder}  is sharp in the following sense.
\begin{example*}[Proposition \ref{pro:sharp}]
  For any $\epsilon > 0$, there exists some $\alpha > 0$ and a curve $\gamma \in C^{(1 - \epsilon)\alpha}(\R_+; \R^2)$  such that the solution to the ODE
  $$\dd x_t =  \|x_t\|^{1 + \alpha}J  x_t \; \dd t + \dd \gamma_t, \qquad J=\begin{pmatrix}0 &-1\\ 1 &0\end{pmatrix} $$
  explodes in finite time. 
  
Note that the drift term $b(x) = \|x\|^{1 + \alpha} J x$  belongs to $ \Lip_{\loc}(\R^2; \R^2)$ and satisfies 
the assumptions in Proposition \ref{cor:B} with $f(s)=s$ and $\beta = 1 + \alpha$.
\end{example*}

Our approach allows a smooth transition from maximal solutions to uniform non-explosion for SDEs driven 
by noise of mixed types, with finite Hölder or finite variational norms. In particular, 
in the case where the driver is given by (fractional) Brownian motions or L\'evy processes, 
we obtain the following result.
\begin{theorem*}\label{thm:additive-main}
Let $b \in \CB(\R_+ \times \R^d; \R^d)$ satisfy the radial linear growth condition with control $f$ (Assumption \ref{as:lin-growth}),  the orthogonal growth~\eqref{eq:orthogonal-growth}, and condition~\eqref{eq:boundedness} where
 $\beta$ is specified below. Consider the SDE
  \begin{equation}\label{eq:SDEgen}
    \dd x_t = b(t, x_t) \dd t + \dd X_t.
  \end{equation}
  Then the following statements hold.
  \begin{longlist}
    \item [(i)]    
          The SDE is strongly complete and path-by-path unique, if $X = W$ is a Brownian motion,
          $\beta < \frac{3}{2}$ and $b \in L^q([0,T]; L^p_{\loc}(\R^d; \R^d))$ for some $p, q > 2$ 
          satisfying $2 / q + d / p < 1$.
    \item [(ii)]
          The SDE is strongly complete and path-by-path unique, if
          $X = B^H$ is a fractional Brownian motion with Hurst index $H \in (0, 1)$, $\beta < 1 + H$,
          and $b \in C^\alpha_{\loc}(\R^d; \R^d)$ (now $b$ is time-homogeneous),
          $\alpha \in (1 - \frac{1}{2H}, 1)$.
    \item[(iii)]
          The SDE has a unique global solution flow, which is furthermore locally Lipschitz in its initial condition, 
          if $b$ is locally Lipschitz and
          $X = Z$ is a L\'evy process with L\'evy triple $(a, \Sigma, \nu)$ for which 
          $$\int_{\R^d} \left(1 \wedge \|z\|^p\right) \nu(\dd z) < \infty$$ for $p = 2$ and $\beta < \frac{3}{2}$; 
          or if  $\Sigma = 0$, $p \in (1, 2)$,  and $\beta < 1 + \frac{1}{p}$.
  \end{longlist}
\end{theorem*}
This theorem recovers and strengthens \cite[Proposition 7.1]{Scheutzow:17} and certain results in \cite{Friz-Zhang}.

\subsection{Organization}
This article is structured as follows:
\begin{itemize}
  \item Section~\ref{sec:ODE} studies ODEs perturbed by a Hölder continuous curve, laying the foundation for the rest of the work. We obtain bounds on the escape rate of solutions, show the sharpness of these bounds, and establish strong completeness for SDEs with additive noise.
  \item Section~\ref{sec:RDE} develops a non-explosion criterion for RDEs, beginning with YDEs. As a consequence, we derive a strong completeness result for SDEs with multiplicative noise.
  \item Appendix~\ref{sec:cont-holder} contains the proof of Lemma~\ref{lem:cont-holder}, a technical tool used in Section~\ref{sec:RDE}.
  \item Appendix~\ref{sec:extension} shows that maximal solutions can be obtained from local solutions without assuming uniqueness.
\end{itemize}

\subsection{Notations}

Throughout this paper,  $V$ and $W$ denote Banach spaces, and $H$ denotes a Hilbert space. We write
$\CL(V; W)$ for the space of bounded linear operators from $V$ to $W$, and $\CB(V; W)$ for the
set of Borel measurable functions from $V$ to $W$. The space $\Lip(V; W)$ and
$\Lip_{\loc}(V; W)$ consist of globally and locally Lipschitz continuous functions from $V$ to $W$ respectively.
 
The notation $f(x) \lesssim g(x)$ means that there exists a constant $C>0$ such that
$f(x)\le Cg(x)$ for all~$x$. For $f : V \to W$, we write $Df$ for its 
Fr\'echet derivative.  

We write \(L^q([0,T]; L^p(V; W))\) for the space of functions \(f : [0,T] \times V \to W\) such that
$$\int_0^T \left(\int_V \|b(t, x)\|^p \dd x\right)^{\frac{q}{p}} \dd t < \infty.$$
We define \(L^q([0,T]; L^p_{\loc}(V; W)) = \bigcap_{K} L^q([0,T]; L^p(K; W)),\)
where the intersection is taken over all compact sets \(K \subseteq V\).

\medskip

{\it H\"older and $p$-variation norms.} For $\alpha \in (0, 1]$ and $I \subseteq [0, T]$,
we denote $$C^\alpha_I = C^\alpha(I; V)$$ for the space of $\alpha$-Hölder continuous functions 
from $I$ to $V$.

For $f \in C^\alpha_I$, we write 
$$\|f\|_{\alpha; I} = \sup_{s \neq t \in I} \frac{\|f(t) - f(s)\|}{|t - s|^\alpha}$$
for its H\"older semi-norm on $I$, and $\|f\|_{\infty; I}=\sup_{t\in I}\|f(t)\|$ its supremum norm on $I$.
We omit the subscript $I$ when $I = [0, T]$, $\R_+$ or when $I$ is clear from context. 
Finally, we set $C^\alpha_{\loc} = \bigcap_K C^\alpha_K$ where the intersection is taken over all compact sets $K$.

 For $p \ge 1$, we denote $$C^{\pvar}_I = C^{\pvar}(I; V)$$ for the space of functions with finite $p$-variation, and $\|f\|_{\pvar; I}$ for the $p$-variational semi-norm of $f$ on the interval $I$,
i.e.
$$\|f\|_{\pvar; I} = \left(\sup_{P \in \mathcal{P}_I} \sum_{(t_i, t_{i + 1}) \in P} \|f(t_{i + 1}) - f(t_i)\|^p\right)^{\frac{1}{p}}$$
where $\mathcal{P}_I$ denotes the set of all partitions of $I$. 
As before, we write $C^{\pvar}_{\loc} = \bigcap_K C^{\pvar}_K$ with
the intersection taken over all compact sets $K$. 

\medskip

{\it Rough Path Spaces. } 
We recall some notation from rough path theory that will be used in the second part of this article.
For $T>0$, set
 $$\Delta_T = \{(s, t) \in [0, T]^2 : s \le t\}; \quad \Delta = \{(s, t) \in \R_+^2 : s \le t\}.$$
For $\alpha > 1$, define
\begin{equation}
  C^{\alpha}([0, T]; V) = \left\{f \in C^{\lfloor \alpha \rfloor}([0, T]; V) : \|f\|_{\alpha} <\infty\right\}
\end{equation}
where
\begin{equation}\label{eq:holder_norm}
  \|f\|_{\alpha} = \|D^{\lfloor \alpha \rfloor} f\|_{\alpha - \lfloor \alpha \rfloor}
  + \sum_{\substack{\ell\in\N_0\\|\ell|\leq\lfloor\alpha\rfloor}} \sup_{t \in [0, T]}\|D^{\ell}f(t)\|.
\end{equation}
For a two parameter process $A : \Delta_T \to V$ and $s < t \in [0, T]$, 
its $\alpha$-H\"older norm on $[s, t]$ is
$$\|A\|_{\alpha; [s, t]} = \sup_{u < v \in [s, t]} \frac{\|A_{u, v}\|}{|u - v|^\alpha}.$$
We denote by $\C^\alpha(\Delta_T; V)$ the space of two parameter processes with finite $\alpha$-H\"older norm.

For $s < u < t \in [0, T]$, set $$\delta A_{s, u, t} = A_{s, t} - A_{s, u} - A_{u, t}.$$
Finally, for any path $x : [0, T] \to V$, we write
 $$x_{s,t} = x_t - x_s, \qquad s \le t \le T.$$

\section{Background}\label{sec:str-comp}
\subsection{Strong completeness}

In the absence of a diffusion term ($\sigma=0$), the SDE \eqref{eq:sde-intro} reduces to a deterministic
ODE. In this case, for sufficiently smooth $b$, the strong completeness property is equivalent to
completeness.  For a long time, it was widely assumed that the same holds for SDEs, that is a complete SDE with smooth coefficients would automatically be strongly complete. This belief remained unchallenged until K. D. Elworthy provided the first counterexample in \cite{Elworthy:78}.
\begin{example}[\cite{Elworthy:78}]
  Consider the SDE on $\R^2$ 
  \begin{equation}\label{eq:elworthy}
    \begin{split}
      \dd x_t & = (y_t^2 - x_t^2) \dd W_t^1 + 2x_t y_t \dd W_t^2,  \\
      \dd y_t & = -2x_t y_t \dd W_t^1 + (x_t^2 - y_t^2) \dd W_t^2.
    \end{split}
  \end{equation}
  Using a Lyapunov function, one can show that this SDE is complete. However, strong completeness fails. Indeed, under a suitable diffeomorphism, the system transforms into the SDE
  $$\dd z_t = \dd W_t \quad  \hbox{ on } \quad \R^2 \setminus \{0\}, $$ for which \eqref{eq:elworthy} explodes if
  and only if $W_t(\omega)$ hits $-z_0$.
\end{example}

This motivates the search for general criteria ensuring strong completeness for SDEs. In light of ODE 
theory and the above example, one might suspect that the problem stems from the super-linear growth 
of the coefficients. However, Li and Scheutzow \cite{Li:09} demonstrated that this is not the whole story. There, 
they constructed an SDE of the form $\dd x_t = \sigma(x_t)\dd W_t$, where $\sigma$ is smooth and bounded, 
which nevertheless fails to have uniform non-explosion, and consequently is not strongly complete.

It might seem natural to prove uniform non-explosion by first establishing non-explosion for a countable dense set of initial conditions and then extending the result to all initial conditions via continuity. However, this strategy does not always succeed: even for ODEs, the explosion time need not depend continuously on the initial condition.

\begin{example}
  Consider the ODE $$z_t' =z^2_t$$ on the complex plane $\bf{C}$. Its solution is
  $z_t =\f {z_0}{1-z_0 t}$ so the explosion time is $\zeta(z_0) =\f 1 {z_0}$ for any real 
  initial condition $z_0$ and $\zeta(z_0)=\infty$ otherwise. Thus, the explosion time is not 
  continuous in $z_0$. The same discontinuity phenomenon occurs for SDEs; for example, consider  $\dd z_t=z_t^2\circ \dd W_t$.
\end{example}

For SDEs on complete, non-compact Riemannian manifolds, a criterion for strong completeness was 
established by Li in her Ph.D. thesis \cite{Li:thesis}. The method couples the SDE with its derivative 
flow (the solution of the linearized equation along the original trajectories) and derives strong completeness 
by controlling the growth of the vector fields and their derivatives. 
We shall briefly explain the results in \cite{Li:94a} applied $\R^d$ for which the most relevant examples are
collected into Corollary \ref{cor:strong-comp-old}.

\subsubsection{Strong completeness via linearization}
Assume the vector fields $\sigma$ and $b$ are sufficiently regular (e.g. $\sigma, b\in C^2$) so that the
SDE~\eqref{eq:sde-intro} is locally well-posed. Let $\phi_t$ denote its maximal solution flow and set 
$x_t=\phi_t(x_0)$ for the solution starting from $x_0$.

Consider the linearized equation with initial condition~$v_0$,
\begin{equation}\label{eq:deriv-sde}
  \dd v_t = D b(x_t) v_t \,\dd t + \sum_{k=1}^m D\sigma_k(x_t)v_t \,\dd W_t^k.
\end{equation}
 When \eqref{eq:sde-intro} and \eqref{eq:deriv-sde} are regarded as a coupled system, the process $(v_t)$ is the derivative (in probability) of $\phi_t$ with respect to its initial data in the direction $v_0$. This probabilistic differentiability is weaker than almost sure differentiability; the latter holds under strong completeness when the coefficients are sufficiently regular (e.g. $C^2_b$).

Denote the derivative flow, the solution flow of~\eqref{eq:deriv-sde},  by
 $$\{D\phi_t(x_0)v_0: (x_0, v_0)\in \R^d\times \R^d\}.$$
  Li \cite{Li:94a} showed that if the SDE is complete at one point (namely, there exists some $x_0$ 
such that the solution starting from $x_0$ exists for all time almost surely) and
  $$\sup_{x\in K}\E\left[\sup_{s\le t} \|D \phi_s(x)\|^{q}\1_{s<\zeta(x)}\right] < \infty,$$
 for all compact $K\subset\R^d$ and some $q>d-1$, then the SDE is strongly complete. 
 
This control can be obtained via bounds on the derivatives of $\sigma_k$ and $b$ through a function  
$$H_q(x, v)= 2\langle D b(x)v, v\rangle
  +(q-2) \sum_{k=1}^m \f { \langle D \sigma_k(x)v, v\rangle ^2}{\|v\|^2}+\sum_{k=1}^m\|D\sigma_k(x)v\|^2.$$
By It\^o's formula, 
\begin{equation}
  \|v_t\|^q = \|v_0\|^q N_t \exp\left(\f q 2\int_0^t H_q(x_s, v_s) \dd s\right)
\end{equation}
where  $N_t$ is the exponential martingale given below,
$$N_t= \exp\left(\sum_{k=1}^m \int_0^t\f {\<D \sigma_k(x_s)v_s ,v_s\>}{\|v_s\|^2} \dd W_s^k
  -\f 12 \sum_{k=1}^m \int_0^t  \f {\<D \sigma_k(x_s) v_s,v_s\>^2} {\|v_s\|^4}\dd s\right).$$
  
\begin{theorem}\cite[Theorem 4.1, 5.1]{Li:94a}
  Let $\sigma, b\in C^2$. Suppose \eqref{eq:sde-intro} is complete at one point. If there exists
  $g: \R^d\to \mathbb \R_+$ such that, for any compact set $K\subset \R^d$ and any $t>0$,
  $$\sup_{x \in K} {\mathbb E}\left(e^{6q^2 \int_0^{t} g(\phi_s(x))  {\bf 1}_{s<\zeta(x)} \dd s}\right)<\infty$$
  for some $q > d - 1$, and if, all $x$
  \begin{equation}\label{eq:deriv-cond}
    \sum_{k=1}^m\|D\sigma_k(x)\|^2 \le g(x)\quad \text{ and }\quad \sup_{\|v\| = 1} H_q(x,v)\le 6 q \,g(x).
  \end{equation}
then the SDE \eqref{eq:sde-intro} is strongly complete.  If these conditions hold for both $b$ and $-b$, the flow is a.s. a homeomorphism.
\end{theorem}

If there exists a smooth function $g$ whose derivatives satisfy the inequality \eqref{eq:deriv-cond2} below, then the exponential moment $\E[ e^{c g(x_t)}]$ grows at most exponentially in $t$. This  Lyapunov-type estimate controls the growth of  $g(x_t)$ and can be used to verify the condition in the theorem above.

\begin{proposition}\label{lem:grad-bd}\cite[Lemma 6.1]{Li:94a}
  Let $c>0$ and $g \in C^2(\R^d; \R)$. If for some $C>0$,
  \begin{equation}\label{eq:deriv-cond2}\frac12 \sum_{k = 1}^m \<D g, \sigma_k\>^2
    + \frac12 \sum_{k = 1}^m D^2 g (\sigma_k,\sigma_k)+\<D g, b\>\le C,
  \end{equation}
  where $D^2 g$ denotes the Hessian of $g$, then $$\mathbb E[ e^{c g(x_t)}]\le e^{Kt} e^{cg(x_0)}.$$

  Consequently, \eqref{eq:sde-intro} is strongly complete provided it is complete and both \eqref{eq:deriv-cond} and \eqref{eq:deriv-cond2} hold.
\end{proposition}

As an illustration, applying the above proposition with the choices
$g(x)=c(1+\|x\|^2)$, $g = c$, $g(x) = c\log (1 + \|x\|^2)$, $g(x)=c(1+\|x\|^2)^\epsilon$ and
$g(x)=ce^{1+\|x\|^2}$ for some constant $c>0$, yields the following concrete criteria.
\begin{corollary}\label{cor:strong-comp-old}\cite[Theorem 6.2 \& Corollary 6.3]{Li:94a}
  Assume that $\sigma, b\in C^2$. Then \eqref{eq:sde-intro} is strongly complete if one of the
  following conditions holds for all
  $k=1,\dots, m$, $x, v\in \R^d$.
  \begin{enumerate}
  \item[\textnormal{(1)}] \textbf{Bounded coefficients:}  
    \[
    \|D\sigma_k(x)\|^2 \lesssim 1+\|x\|^2, \quad
    \langle D b(x)v, v\rangle \lesssim 1+\|x\|^2,
    \]
    and
    \[
    \langle x, b(x)\rangle + \sum_{k=1}^m\|\sigma_k(x)\|^2
      + \sum_{k=1}^m \langle x, \sigma_k(x)\rangle^2 \lesssim 1.
    \]

  \item[\textnormal{(2)}] \textbf{Bounded derivatives:}  
    \[
    \|D\sigma_k(x)\|^2 \le C, \quad
    \langle D b(x)v, v\rangle \le C,
    \]
    for some constant $C>0$.

  \item[\textnormal{(3)}] \textbf{Linear growth:}  
    \[
    \|D \sigma_k(x)\|^2 \lesssim 1+\log(1+\|x\|^2), \quad
    \langle D b(x)v, v\rangle \lesssim 1+\log(1+\|x\|^2),
    \]
    and
    \[
    \langle x, b(x)\rangle + \sum_{k=1}^m\|\sigma_k(x)\|^2
      + \sum_{k=1}^m \langle x, \sigma_k(x)\rangle^2 \lesssim 1+\|x\|^2.
    \]

  \item[\textnormal{(4)}] \textbf{Polynomial growth/decay:}  
    For some $\epsilon \in (0,\infty) \setminus \{1\}$,
    \[
    \|D\sigma_k(x)\|^2 \lesssim (1+\|x\|^2)^\epsilon, \quad
    \langle D b(x)v, v\rangle \lesssim (1+\|x\|^2)^\epsilon,
    \]
    and
    \[
    \langle x, b(x)\rangle + \sum_{k=1}^m \|\sigma_k(x)\|^2
      + 2(\epsilon-1) \sum_{k=1}^m\frac{\langle x, \sigma_k(x)\rangle^2}{1+\|x\|^2}
      \lesssim (1+\|x\|^2)^{1-\epsilon}.
    \]

  \item[\textnormal{(5)}] \textbf{Exponential decay at infinity:}  
    \[
    \|D\sigma_k(x)\|^2 \lesssim e^{1+\|x\|^2}, \quad
    \langle D b(x)v, v\rangle \lesssim e^{1+\|x\|^2},
    \]
    and
    \[
    \langle x, b(x)\rangle + 2\sum_{k=1}^m \|\sigma_k(x)\|^2
      + 2\sum_{k=1}^m \frac{\langle x, \sigma_k(x)\rangle^2}{1+\|x\|^2}
      \lesssim e^{-\|x\|^2}.
    \]
\end{enumerate}
\end{corollary}

Recent work on strong completeness of SDEs has focused on growth assumptions on the drift 
 and lowering regularity requirements for the coefficients. In particular, when $\sigma, b\in C^1$,  Fang, Imkeller, and Zhang \cite{Fang:07} showed that part (3) of Corollary \ref{cor:strong-comp-old} remains valid.  Scheutzow and Schulze \cite{Scheutzow:17} further improved the results of Li \cite{Li:94a} by removing the derivative condition on $b$ and replacing it with a one-sided local Lipschitz condition.  In the additive noise case, they also proved that
 $$\dd x_t=b(x_t)dt+\dd W_t$$ is strongly complete if $b : \R^d\to \R^d$ has
at most radial linear growth with at most quadratic rotation. 

For SDEs satisfying a non-uniform ellipticity condition, Chen and Li \cite{Chen:14} reduced the regularity requirements on the coefficients to Sobolev classes and established the existence of a global flow of Sobolev class.

For non-regular drifts in the additive noise case, Krylov and Röckner \cite{Krylov:05} established well-posedness of solutions in the $L^q([0,T]; L^p(\R^d; \R^d))$-framework and proved continuity in probability with respect to the initial condition. Fedrizzi and Flandoli \cite{Fedrizzi:11} provided a new proof and strengthened the result, while  Anzeletti, L\^e, and Ling \cite{Anzeletti:23} established path-by-path uniqueness, stability, and strong completeness. By a localization argument, we extend these results to unbounded drifts (see Proposition \ref{thm:additive-ext}) in the local integrability framework. 


For (not necessarily strictly) elliptic SDEs, strong completeness has also been established in Zhang \cite{Zhang-16}, Xie and Zhang \cite[Theorem. 1.2]{Xie-Zhang-16}, and Ling, Scheutzow, and Vorkastner \cite[Theorem. 4.8]{Ling:21}, under the assumption that $b$ and $D \sigma_k$ are locally integrable, together with additional structural conditions.

Complementing the the Sobolev-flow result in Chen and Li \cite{Chen:14}, Jentzen, Kuckuck, Muller, Gronbac, and Yaroslavtseva \cite{Jentzen-Kuckuck-Muller-Gronbach-Yaroslavtseva-22} constructed additive SDEs with smooth drifts whose derivatives grow at most polynomially, but whose solution maps are not locally Lipschitz, nor even locally Hölder in their initial values.
 In contrast, Jetzen and Kuckuck \cite{Jentzen21} showed that any additive SDE satisfying a Lyapunov-type condition (ensuring existence and uniqueness) and having a drift with first derivatives of at most polynomial growth is at least logarithmically Hölder continuous in its initial data.
Finally, in the SPDE context, strong completeness has been studied in Marinelli, Pr\'ev\^ot, and R\"ockener \cite{Marinelli-10}.

\subsubsection{Differentiability in \texorpdfstring{$L^p$}{Lp}}

Closely related is the concept of differentiability in $L^p$ with respect to the initial data,
which often serves as a starting point for applying Kolmogorov's criterion to obtain continuity in the initial point.

However, the solution of a complete SDE with smooth coefficients need not be continuous in $L^2$,
let alone differentiable. For example, consider the drift vector field $b(x_1, x_2)=(x_1x_2, -x_1^2)$
and the two dimensional SDE $\dd x_t=b(x_t)\dd t+ x_t\dd W_t$.  Hairer, Hutzenthaler, and Jetzen \cite{Hairer:15} showed that the solution flow to this SDE is not locally Lipschitz continuous in $L^2$ with respect to its initial conditions. By contrast, Cox, Hutzenthaler, and Jentzen \cite{cox:24} proved that Lipschitz continuity in $L^2$ does hold for $\dd x_t=b(x_t)\dd t+\sigma(x_t)\dd W_t$
when $\sigma$ is bounded and Lipschitz continuous.

In a related direction, Chak \cite{Chak} studied the existence of twice differentiable solutions to Kolmogorov equations and weak convergence rates of order one for numerical approximations.
The $L^p$ continuity of solutions with respect to the initial condition has also been investigated by 
Pr\'ev\^ot and R\"ockner \cite{Prevot-Rockner} for both SDEs and SPDEs.

\subsection{Rough differential equations}

An RDE with sufficiently smooth coefficients admits a local solution from any initial point (see, for example, the book by Friz and Hairer \cite{Friz:20}). 
In the linear case, global existence was shown in Lyons \cite{Lyons:94}. For nonlinear coefficients, one typically requires that the coefficients are bounded and have bounded derivatives in order to ensure global existence. 

For unbounded coefficients, global existence theorems for RDEs
of the form $\dd x_t =\sigma(x_t) \dd \mathbf{X}_t$ were studied in Davie \cite{Davie:07} using a discrete approximation (Euler Scheme) of the solution.
Lejay \cite{Lejay:09} established global existence under the assumptions that $\sigma$ satisfies a sufficiently slow growth condition and has a bounded derivative that is Hölder continuous of sufficiently high order. Later, in Lejay \cite{Lejay:12}, using again an Euler scheme, he proved a global existence theorem under the conditions that $D \sigma$ is bounded and that $D \sigma(\cdot) \sigma(\cdot)$ is globally H\"older continuous.
A concise proof of this result was recently presented in the lecture notes  Caravanenna and Zambotti \cite{Caravenna:25} in the special case where both $D \sigma$ and $D(D \sigma(\cdot),\sigma(\cdot))$ are bounded. A related result can be found in Bailleul and Catellier \cite{Bailleul-Catellier} where bounded derivatives are also assumed.

Finally, in the case where a drift term is present in the RDE (as in our setting), i.e.,
$\dd x_t = b(x_t) \dd t + \sigma(x_t) \dd \mathbf{X}_t$, 
a global existence result was established by Riedel and Scheutzow \cite{Riedel:16} in the weakly geometric case, assuming linear growth of the drift term and $\sigma \in C^{3+}_b$ (so that derivatives up to order $3$ are bounded). In the context of branched rough paths, a special case of an RDE with 
super-linear inward-pointing drift was studied by  Bonnefoi, Chandra, Moinat, and Weber \cite{bonnefoi2022}.

\section{ODEs with forcing and additive SDEs}\label{sec:ODE}

Throughout this section, let $\gamma : \R_+ \to H$ and $b : \R_+ \times H \to H$ be measurable mappings. 

\subsection{Lower bound on the escape rate of ODEs with forcing}

Lemma \ref{lem:fast-growth} provides a useful bound, used repeatedly in this paper, on the time required for a deterministic process to reach specified level sets.


\begin{definition}\label{def:sol}
  We consider the ODE with driving curve $\gamma$
  \begin{equation}\label{eq:ode-intro}
    \dd x_t = b(t, x_t)\dd t + \dd \gamma_t.
  \end{equation}
  A solution with initial condition $(s, x) \in \R_+ \times H$ is a pair $(x_t, \xi)$ such that
  \begin{equation}\label{integral}
      x_t = x + \int_s^t b(s, x_r) \dd r + \gamma_t - \gamma_s, \qquad \forall t \in [s, \xi).
    \end{equation}
  It is called a maximal solution if  $\xi < \infty$ implies that $\limsup_{t \uparrow \xi} \|x_t\| = \infty$.
  The solution is global if $\xi = \infty$.
\end{definition}
\begin{assumption}\label{cond:perp}
 Let $b$ satisfy Assumption~\ref{as:lin-growth} with control $f$, and suppose:
  \begin{longlist}
    \item[(A.1)] (bounded at the origin) 
    There exists $a > 0$ such that, for all $T > 0$,
      \begin{equation}\label{eq:bound-at-zero}
        b_a^T := \sup_{t \in [0,T]} \sup_{\|x\| \le a} \|b(t, x)\| < \infty.
      \end{equation}
    \item[(A.2)] (orthogonal control) 
    There exists some $\beta>0$, such that 
      \begin{equation}\label{eq:perp}
        \left|\left\langle y, b(t, x)\right\rangle\right| \le (1 + \|x\|)f(\|x\|)^\beta,
      \end{equation}
      for all $x, y, t$ with $x \perp y$ and $ \|y\| = 1$.
  \end{longlist}
\end{assumption}
We remark that condition~\eqref{eq:perp} can be equivalently stated as 
$$\left\|\Proj_{x^\perp} b(t, x) \right\| \le (1 + \|x\|) f(\|x\|)^\beta, \qquad \text{ for all } x, t,$$
where $\Proj_{x^\perp}$ denotes the projection onto  $\Span{\{x\}}^\perp$.
Condition~\eqref{eq:bound-at-zero} rules out singularities at the origin. For example,
\(x \mapsto -\frac{\mathbb{1}_{\{x \neq 0\}}}{\|x\|^2}x\) satisfies both 
Assumption~\ref{as:lin-growth} and condition~\eqref{eq:perp}, but is excluded here.

%


\begin{lemma}\label{lem:ode-est}
Let $b$ satisfy Assumption \ref{cond:perp} and let $\gamma$ satisfy $\gamma_0 = 0$. 
  If $(x_t)_{t \in [0, \xi)}$ is a solution to \eqref{eq:ode-intro} with initial condition
  $(0, x_0)$, where $x_0 \in H$, then for any $T \in [0, \xi)$,
  \begin{align*}
    \|x_T - \gamma_T\|^2 \le & \|x_0\|^2 + 2 \int_0^T f(\|x_t\|)\left(\|x_t\| + \|\gamma_t\|\right) \dd t \\
                             & + 2 \int_0^T \|\gamma_t\|
    \left(\sup_{\substack{\|x\| \le \|\gamma_t\|}}\|b(t, x)\| + (1 + \|x_t\|) f(\|x_t\|)^\beta\right) \dd t.
  \end{align*}
\end{lemma}
\begin{proof}
  In the case where $t \le T$ and $x_t \neq 0$, we set 
  $\alpha_t = \<\gamma_t, \frac{x_t}{\|x_t\|^2}\>,$
  so that $\alpha_t x_t$ is the projection of $\gamma_t$ onto the direction of $x_t$. Defining
  $$y_t = x_t - \gamma_t = x_0 + \int_0^t  b(s, x_s) \dd s,$$
  we compute
  \begin{align*}
    \dv{\|y_t\|^2}{t}\bigg|_{t = t} & = 2 \<x_t - \gamma_t, b(t, x_t)\>                                               \\
                                    & = 2 (1 - \alpha_t)\<x_t,  b(t, x_t)\> - 2\<\gamma_t - \alpha_t x_t, b(t, x_t)\>.
  \end{align*}
  By considering separately the cases $\alpha_t \le 1$ and $\alpha_t > 1$, and using the growth
  bounds from the assumption, the first term is bounded by
  $$2 (1 - \alpha_t)\<x_t,  b(t, x_t)\> \le 2(\|x_t\| + \|\gamma_t\|) f(\|x_t\|) +
    2 \|\gamma_t\| \sup_{\|x\| \le \|\gamma_t\|} \|b(t, x)\|.$$
  Here, we used the fact that $\|x_t\| < \|\gamma_t\|$ for $\alpha_t > 1$. 
  
  When $\alpha_t \le 1$, we have $1 - \alpha_t \ge 0$, and noting that 
  $\|\alpha_t x_t\| \le \|\gamma_t\|$, applying \eqref{cond:cos} gives
    \begin{align*}
    (1 - \alpha_t)\<x_t,  b(t, x_t)\> & = (\|x_t\| - \alpha_t \|x_t\|) \<\frac{x_t}{\|x_t\|},  b(t, x_t)\> \\
                                      & \le (\|x_t\| + \|\gamma_t\|) f(\|x_t\|).
  \end{align*}
  Moreover, since $\gamma_t - \alpha_t x_t $ is orthogonal to $ x_t$, the second term is bounded by
  $$- 2\<\gamma_t - \alpha_t x_t, b(t, x_t)\> \le
    2\|\gamma_t\| (1 + \|x_t\|) f(\|x_t\|)^\beta.$$
  Thus, we obtain the inequality
  \begin{equation}\label{eq:diff-y}
    \begin{aligned}
      \dv{\|y_t\|^2}{t}\bigg|_{t = t}
      \le & 2\|\gamma_t\|\left((1 + \|x_t\|)f(\|x_t\|)^\beta
      + \sup_{\|x\| \le \|\gamma_t\|} \|b(t, x)\|\right)     \\
          & + 2f(\|x_t\|)\left(\|x_t\| + \|\gamma_t\|\right).
    \end{aligned}
  \end{equation}

In the case $x_t = 0$, we have
  $$\dv{\|y_t\|^2}{t}\bigg|_{t = t} = 2\<\gamma_t, b(t, 0)\>
    \le 2\|\gamma_t\| \sup_{\|x\| \le \|\gamma_t\|} \|b(t, x)\|$$
so the same inequality,  \eqref{eq:diff-y}, still holds. Since $\|y_0\|^2 = \|x_0\|^2$, 
integrating both sides of this differential inequality, yields the desired bound.
\end{proof}

The above lemma will allow us to control the size of $x_t$ in terms of the volatility of $\gamma$.
To this end, we introduce the following notion of interactive regularity for $\gamma$.

\begin{definition}\label{def:interactive}
  For $T > 0$, denote $\tau^r = \inf \{t : \|x_t\| \ge r\} \wedge T$. 
  We say that $\gamma$ satisfies an interactive relation at level $k \in \mathbb{N}$ if
  \begin{equation}\label{eq:local-holder}
    \sup_{0 \le t \le \sigma^k \wedge \delta_k}\|\gamma_{\tau^{Rk} + t} - \gamma_{\tau^{Rk}}\|
    \le K \delta_k^{\beta - 1},
  \end{equation}
  where $\sigma^k := \tau^{R(k + 1)} - \tau^{Rk}$
  for some $\beta \in (1, 2)$ and positive numbers $K$, $R$, $\delta_k$. 
\end{definition}
Here, $K$ can be interpreted as a weak form of the $(\beta - 1)$-H\"older norm of $\gamma$. 
The quantities $\delta_k$ will serve as lower bounds for the time required for a solution to the ODE to cross each level.
Thus, we expect non-explosion if the sequence $(\delta_k)$ can be chosen to be not-summable. We also want to consider a number $R$ such that~\eqref{eq:local-holder} holds uniformly over $k$. For our application, it is sufficient if $R$ can be taken larger than $R_0$ where
\begin{equation}\label{eq:R0-def}
  R_0 = (2K + 1) + \sqrt{(2K + 1)^2 + 2K(2 + b^T_a)},
\end{equation}
and \(b^T_a\) is the constant appearing in Assumption~\ref{cond:perp}.

\begin{lemma}\label{lem:fast-growth}
  Let $(x_t)_{t \in [0, \xi)}$ be a maximal solution to \eqref{eq:ode-intro} and $T > 0$.
  Assume that
  \begin{enumerate}
    \item[(i)]\label{lem:fast-growth1} $b$ satisfies Assumptions~\ref{cond:perp} with $\beta \in (1, 2)$ and control $f$.
        \item[(ii)]\label{cond:gamma-reg} The interactive relation~\eqref{eq:local-holder} holds 
    with $\beta$ from \rm{(i)} at some level $k \in \mathbb{N}$, with the constants in~\eqref{eq:local-holder} satisfying
      \begin{equation}
        \delta_k \le \min\{1, f(R(k + 1))^{-1}, (a / K)^{\frac{1}{\beta - 1}}\},
      \end{equation}
      and $R > R_0 \vee \|x_0\|$  and $R_0$ is defined by~\eqref{eq:R0-def}.
  \end{enumerate}
  Then, denoting $\sigma^k$ and $\tau^{Rk}$ as in Definition~\ref{def:interactive}, we have 
  $$\sup_{0 \le s \le \delta_k \wedge \sigma^k \wedge T} \|x_{s + \tau^{Rk}}\| < R(k + 1).$$
In other words, $\tau^{R(k + 1)} - \tau^{Rk} > \delta_k$ whenever $\tau^{R(k + 1)} \le T$.

  Furthermore, if there exists some $k_0 \in \N$ such that $\sum_{k = k_0}^\infty \delta_k > T$
  and the above assumptions holds with constants $K, R$ chosen uniformly for all $k \ge k_0$, then 
  the solution exists at least up to time $T$, that is, $\xi > T$.
\end{lemma}
\begin{proof}
  We shall show that $$\sup_{0 \le s \le \delta_k \wedge \sigma^k \wedge T} \|x_{s + \tau^{Rk}}\| < R(k + 1)$$
  whenever the following quadratic inequality is satisfied:
  \begin{equation}\label{eq:R-bd}
    (Rk)^2 + 2(K + 1)R(k + 1) + 2K(2 + b^T_a) \le (R(k + 1) - K)^2.
  \end{equation}
  Here $b^T_a$ is the bound in Assumptions~\ref{cond:perp}.

  We further observe that for any $k \in \N$,   Equation~\eqref{eq:R-bd} is equivalent to
  $$R^2 - 2(2K + 1)\frac{k + 1}{2k + 1}R + \frac{K^2 - 2K(2 + b^T_a)}{2k + 1} \ge 0,$$
  which is satisfied for all  $R > R_0$, where $R_0$ is the positive root of the quadratic equation
  $$R_0^2 - 2(2K + 1)R_0 - 2K(2 + b^T_a) = 0.$$

  We proceed to prove the upper bound for  $\sup_{0 \le s \le \delta_k \wedge \sigma^k \wedge T} \|x_{s + \tau^{Rk}}\|$.
  As the conclusion is trivial if $\inf\{t : \|x_t\| \ge kR\} \ge T$, we may assume that $\tau^{Rk} < T$.
  Define
  $$y_s = x_{s + \tau^{Rk}} - (\gamma_{s + \tau^{Rk}} - \gamma_{\tau^{Rk}}).$$
  We shall control its norm via the previous lemma.

  Replacing $x_t$ by $x_{t + \tau^{Rk}}$ and $\gamma_t$ by $\Delta \gamma_t := \gamma_{t + \tau^{Rk}} - \gamma_{\tau^{Rk}}$,
Lemma~\ref{lem:ode-est} implies that, for all $s \le \delta_k \wedge \sigma^k \wedge T$,
  \begin{alignat*}{2}
    \|y_s\|^2
    \le & \|x_{\tau^{Rk}}\|^2 &  & + 2\int_{\tau^{Rk}}^{\tau^{Rk} + s} f(\|x_t\|)\left(\|x_t\| + \|\Delta \gamma_t\|\right) \dd t \\
        &                     &  & + 2 \int_{\tau^{Rk}}^{\tau^{Rk} + s} \|\Delta \gamma_t\|
    \left(b^T_a + (1 + \|x_t\|)f(\|x_t\|)^\beta \right)\dd t,
  \end{alignat*}
  where we had used the assumption that $\|\Delta \gamma_t\| \le K \delta_k^{\beta - 1} \le a.$
Since $s \le \delta_k$, we obtain
  \begin{align*}
    \|y_s\|^2 \le & (Rk)^2 + 2f(R(k + 1))(R(k + 1) + K \delta_k^{\beta - 1})\delta_k                 \\
                  & + 2K \delta_k^{\beta} b^T_a + 2K \delta_k^{\beta}(1 + R(k + 1))f(R(k + 1))^\beta \\
    \le           & (Rk)^2 + 2R(k + 1) + 2K + 2 K b^T_a + 2K (1 + R(k + 1))f(R(k + 1))^\beta,
  \end{align*}
  where in the last line, we have used the bound  $\delta_k \le 1 \wedge f(R(k + 1))^{-1}$. 
  This further simplifies~to
  $$  \|y_s\|^2 \le  (Rk)^2 + 2(K + 1)R(k + 1) + 2K(2 + b^T_a).$$

  Since,
  $$\sup_{0 \le s \le \delta_k \wedge \sigma^k \wedge T} \|x_{s + \tau^{Rk}}\|
    \le \sup_{0 \le s \le \delta_k \wedge \sigma^k \wedge T} \|y_s\| + K,$$
 it follows from \eqref{eq:R-bd}  that
  $\sup_{0 \le s \le \delta_k \wedge \sigma^k \wedge T} \|x_{s + \tau^{Rk}}\| \le R(k+1)$,
  which completes the proof of the first assertion.

  Finally, in the case where $\sum_{k \ge k_0} \delta_k > T$, suppose for contradiction that $\xi \le T$.
Since $\tau^{R(k + 1)} < \xi$ for any $k \ge k_0$, we have 
  $$T \ge \xi = \sum_{k = k_0}^\infty (\tau^{R(k + 1)} - \tau^{Rk}) > \sum_{k = k_0}^\infty \delta_k > T,$$
  which is a contradiction.
\end{proof}

This lemma allows us to obtain an explicit control $\psi^{-1}$ for the growth of $\|x_t\|$.
Taking $\delta_k$ as in Equation~\eqref{eq:delta-k},  define $\psi : [0, T] \to \R_+$ by
$$\psi(N) = \sum_{k \le N + 1} \delta_k, \qquad N \in \N,$$ 
and, for $t \in [0, T] \setminus \N$, define $\psi(t)$
by linear interpolation. Since $\psi$ is continuous and strictly increasing, it is invertible, and we denote its inverse by $\psi^{-1}$.

\begin{corollary}\label{cor:global-holder}
  Let $\alpha \in (0, 1]$ and $\beta \in (1, 1 + \alpha]$.
  Suppose Assumption~\ref{cond:perp} holds with $\beta$, and let $\gamma\in C^{\alpha}_{\loc}(\R_+; H)$.
  Then, any maximal solution $(x_t)_{t \in [0, \xi)}$ of \eqref{eq:ode-intro} is, in fact, a global solution.
  Moreover, for any fixed $T > 0$,
  \begin{equation}\label{eq:sol-bdd}
    \sup_{0 \le s \le t} \|x_s\| \lesssim \psi^{-1}(t), \qquad \forall t \le T
  \end{equation}
 where the implicit constant in $\lesssim$ depends only on $T$ and $\|x_0\|$.
\end{corollary}
\begin{proof}
  Fixing $T > 0$, set $K = \|\gamma\|_{\beta - 1, [0, T]}$, $R =2R_0 \vee \|x_0\|$ with $R_0$ defined by~\eqref{eq:R0-def}.
  Define
  \begin{equation}\label{eq:delta-k}
    \delta_k = \min\{1, f(R(k + 1))^{-1}, (a / K)^{\frac{1}{\beta - 1}}\}.
  \end{equation}
  Non-explosion then follows directly from the second assertion of Lemma~\ref{lem:fast-growth}. Furthermore,  the bound~\eqref{eq:sol-bdd} follows from the fact that $\sup_{s \le \psi(N)}\|x_s\| \lesssim N$ where the constant depends on $R_0$, which is given by $T$.
\end{proof}

\subsection{Applications to SDEs with additive noise}

\begin{definition}
  Let $b : \R_+ \times \R^d \to \R^d$ and $X$ be a stochastic process on $\R^d$. We
  consider the  equation
  \begin{equation}\label{eq:ode-per}
    \dd x_t = b(t, x_t) \dd t + \dd X_t.
  \end{equation}
  \begin{itemize}
    \item We say that \eqref{eq:ode-per} is path-by-path unique if for almost every $\omega$,
          the ODE $$\dd x_t = b(t, x_t) \dd t + \dd X_t(\omega)$$ has a unique maximal solution from
          any initial condition.
    \item We say this equation is strongly complete if for almost every $\omega$,
          there exists a global solution flow $(t, x) \mapsto \phi_t(x, \omega)$ of the ODE
          $\dd x_t = b(t, x_t) \dd t + \dd X_t(\omega)$ which is jointly continuous.
  \end{itemize}
\end{definition}

We note that path-by-path uniqueness is stronger than pathwise uniqueness (i.e. uniqueness of strong
solutions), as it requires uniqueness of solutions that are not necessarily adapted.

Consider a drift $b$ satisfying a Prodi-Serrin condition, namely
 $$b \in L^q([0,T]; L^p(\R^d; \R^d)) \qquad \hbox{with}\; 2 / q + d / p < 1.$$
 Anzeleti, Lê, and Ling \cite[Theorem 2.5]{Anzeletti:23} established 
path-by-path uniqueness and strong completeness under this global condition, generalizing the results of \cite{Krylov:05}. 

In the critical case $2/q + d/p = 1$, the existence of a strong solution was first shown for almost all initial 
conditions \cite{BFGM}. More recently, it was proved to hold for all initial conditions, together
with strong completeness \cite{Rockner-Zhao-2021}. Our results, combined with localizing
the results of \cite{Anzeletti:23}, naturally leads to the following assertion. 

\begin{proposition}\label{thm:additive-ext}
  Let $p, q \in (2, \infty)$ satisfy $2/q + d/p < 1$. If  $b \in L^q([0,T]; L^p_{\loc}(\R^d; \R^d))$ satisfies
  Assumption \ref{cond:perp} with $\beta < \frac{3}{2}$, then the SDE
  \begin{equation}\label{eq:SDEext}
    \dd x_t = b(t, x_t) \dd t + \dd W_t
  \end{equation}
  is both strongly complete and path-by-path unique. 

  If, on the other hand, $d \ge 3$ and $2/q + d/p = 1$, the SDE \eqref{eq:SDEext} is strongly complete.
\end{proposition}
\begin{proof}
  By Example~\ref{cor:loc-KR}, it suffices to show that we have uniform non-explosion 
  which follows directly by applying Corollary~\ref{cor:global-holder}.
\end{proof}

{\it Remark.} In Section 4 we discuss SDEs with multiplicative noise, the counter part of the above Proposition 
in this setting is given in Corollary \ref{cor:strong-complete}.

We also note that under Assumption~\ref{cond:perp}, the drift $b$ is allowed to have singularities 
away from zero in the spatial variable, so the assumption does not imply
\(b \in L^q([0,T]; L^p_{\loc}(\mathbb{R}^d; \mathbb{R}^d))\). 
For example, fix some \(z \neq 0\),  and consider
\[b(t, x) = \begin{cases}
  - \frac{x}{\|x - z\|^d} & \text{if } x \neq z, \\
  0 & \text{if } x = z.
\end{cases}\]
This vector field satisfies Assumption \ref{as:lin-growth} and condition \eqref{eq:perp} for any control function
 $f$, since 
\(\<x, b(t, x)\> \le 0\) and, for any \(y \perp x\),  we have \(\<y, b(t, x)\> = 0\). Moreover,
choosing \(a = \|z\| / 2\),  gives
\[\sup_{t} \sup_{\|x\| \le a} \|b(t, x)\| \le (a / 2)^{1 - d}.\]
 Thus, Assumption~\ref{cond:perp} is satisfied. However, as 
\(\int_{\|x\| \le 2\|z\|} \|b(t, x)\|^p = \infty\) for any \(p \ge 1\), 
\(b \notin L^q([0,T]; L^p_{\loc}(\mathbb{R}^d; \mathbb{R}^d))\). 

\medskip


For mixed additive noise, we have the following result.
\begin{proposition}
  Let $(B^{H_i})_{i = 1}^n$ be a sequence of independent fractional Brownian motions on $\R^d$.
  Denoting $H = \bigvee_{i = 1}^n H_i$ and $h = \bigwedge_{i = 1}^n H_i$,
  let $b \in C^\alpha_{\loc}(\R^d; \R^d)$ with $\alpha > \left(\frac{3}{2} - \frac{1}{2 H}\right)_+$. Then,
  under Assumption~\ref{cond:perp} where $\beta < 1 + h$, the SDE
  \begin{equation}
    \dd x_t = b(x_t) \dd t + \sum_{i = 1}^n \dd B^{H_i}_t,
  \end{equation}
  is strongly complete and path-by-path unique.
\end{proposition}
\begin{proof}
  As before uniform non-explosion follows directly from Corollary~\ref{cor:global-holder} and thus, by localization,
  $b$ can be assumed to be bounded. Then, the conclusion follows from \cite[Theorem 1.12]{Catellier:16}
  where the assumption follows directly from the estimates \cite[Corollary 4.6 and Theorem 3.7]{Catellier:16}.
\end{proof}

In fact, for uniform non-explosion, it is sufficient to assume $\gamma$ has finite $p$-variation for some $p > 1$.
However, in contrast to H\"older continuous driving noises, we no longer have a bound on $\sup_{s \le t} \|x_s\|$.
We remark that, if $\gamma$ is in addition continuous, it can be re-parametrized
so that it is H\"older continuous. Thus, non-explosion follows from the previous corollary. The following
lemma treats the case where $\gamma$ is allowed to have jumps.

\begin{lemma}\label{lem:global-pvar}
  Let $p > 1$ and  $\beta \in (1, 2)$ satisfying $\beta \le 1 + \frac{1}{p}$. Suppose $b$ satisfies
  Assumption~\ref{cond:perp} with the above $\beta$ and $\gamma \in C^{\pvar}_{\loc}(\R_+; H)$.
  Then, any maximal solution of the equation~\eqref{eq:ode-intro}
  is a global solution.
\end{lemma}
\begin{proof}
  It is sufficient to show that $\xi \ge T$ for any $T > 0$ fixed. Let $R$ satisfy \eqref{eq:R-bd}
  for all $k \in \N$, where $K$ is taken to be~$1$. Denote
  $$\delta_k = \min\{1, f(R(k + 1))^{-1}, (a / K)^{\frac{1}{\beta - 1}}\}.$$
  Let $(\delta_{k_n})$ denote the ordered subset of $(\delta_k)$ for which
  $$\sup_{0 \le t \le \delta_{k_n} \wedge (\tau^{R(k_n + 1)} - \tau^{Rk_n})}
    \|\gamma_{\tau^{Rk_n} + t} - \gamma_{\tau^{Rk}}\| \ge \delta_{k_n}^{\beta - 1}.$$
  Suppose $\xi < T$, so $\tau^{Rk} <T$ for all $k \in \N$. Then by Lemma~\ref{lem:fast-growth},
  $\tau^{R(k + 1)} - \tau^{Rk} > \delta_k$ for all $k \notin \{k_n\}_{n \in \N}$. If
  $\sum_{n = 0}^\infty \delta_{k_n} < \infty$, then
  $$\xi = \sum_{k = 0}^\infty (\tau^{R(k + 1)} - \tau^{Rk}) \ge
    \sum_{k \notin \{k_n\}_{n \in \N}} (\tau^{R(k + 1)} - \tau^{Rk})
    \ge \sum_{k \notin \{k_n\}_{n \in \N}} \delta_k = \infty,$$
  where in the last step we used $\sum_{k = 0}^\infty \delta_k = \infty$. This contradicts with
  the assumption $\xi < T$ and we conclude non-explosion.

  It remains to prove that $\sum_{n = 0}^\infty \delta_{k_n} < \infty$. For any $n \in \N$,
  there exists $t_n \le \delta_{k_n} \wedge (\tau^{R(k_n + 1)} - \tau^{Rk_n})$ such that
  $\delta_{k_n} \le\|\gamma_{\tau^{Rk_n} + t_n} - \gamma_{\tau^{Rk_n}}\|^{\f 1 {\beta-1}} + 2^{-n}$.
  Then, denoting $\mathcal{P}^T$ the set of all partitions of $[0, T]$,
  \begin{align*}
    \sum_{n = 0}^\infty \delta_{k_n}
     & \le 1 + \sum_{n = 0}^\infty  \|\gamma_{\tau^{Rk_n} + t_n} - \gamma_{\tau^{Rk_n}}\|^{\f 1 {\beta-1}} \\
     & \le 1 + \sup_{P \in \mathcal{P}^T} \sum_{(t_i, t_{i + 1}) \in P}
    \|\gamma_{t_{i + 1}} - \gamma_{t_i}\|^{\frac{1}{\beta - 1}}
    =  1 + \|\gamma\|_{\frac{1}{\beta - 1}\mathrm{\text{-}var}; [0, T]}^{\frac{1}{\beta - 1}}                              \\
     & \lesssim 1 + \|\gamma\|_{\pvar; [0, T]}^p < \infty,
  \end{align*}
  as required.
\end{proof}

The above criterion for uniform non-explosion for SDEs driven by processes with finite $p$-variation allows us 
to obtain a non-explosion result to equations driven by L\'evy processes.
\begin{proposition}\label{prop:levy}
  Let $Z$ be a L\'evy process in $\R^d$ with L\'evy triple $(a, \Sigma, \nu)$ where
  $\nu(\{0\}) = 0$ and
  $$\int_{\R^d} (1 \wedge \|z\|^p) \nu(\dd z) < \infty.$$
  Let $b \in \CB(\R_+ \times \R^d; \R^d)$ be locally Lipschitz in
  its second argument.  Then the SDE \begin{equation}\label{eq:SDElevy}
    \dd x_t = b(t, x_t) \dd t + \dd Z_t.
  \end{equation}  admits a unique adapted global solution $\phi$ which is furthermore locally Lipschitz in
  its initial conditions, provided one of the following conditions holds
  \begin{enumerate}
    \item [(i)]  $p=2$ and $b$ satisfies Assumption~\ref{cond:perp} for some $\beta < \frac{3}{2}$.
    \item[(ii)] $\Sigma=0 $, $p \in (1, 2)$, and $b$ satisfies Assumption~\ref{cond:perp} for some $\beta < 1 + \frac{1}{p}$.  \end{enumerate}
\end{proposition}
\begin{proof}
  By Example~\ref{cor:loc-levy}, the SDE~\eqref{eq:SDElevy} admits a unique adapted local solution $\phi$
  which is locally Lipschitz in initial conditions. Thus, it suffices to show uniform non-explosion
  which follows directly from Lemma~\ref{lem:global-pvar} and the fact that
  $Z$ has sample paths in $ C^{\pvar}_{\loc}(\R_+; \R^d)$ \cite{Bretagnolle:72}.
\end{proof}

These applications are summarized in Theorem \ref{thm:additive-main}.


The method presented above also allows us to formulate a Montel-Tonelli type condition for
non-explosion. Unlike Lemma~\ref{lem:fast-growth} which considers the time required for the solution flow to move between fixed contours, we now study a family of contours whose centers evolve along the trajectory of a reference solution starting from a point 
 $x_0$ or which no explosion occurs, i.e. $\xi(x_0) = \infty$. By analyzing the evolution of the difference $x_s^x - x_s^{x_0}$,
we obtain conditions ensuring non-explosion for solutions starting from arbitrary initial points.


\begin{example}\label{ex:non-explosion}
  Let $(x_t^x)_{t \in [0, \xi(x))}$ be a maximal solution of \eqref{eq:ode-intro}, with
  initial condition $x~\in~H$. Suppose that there exists some $f : \R_+ \to \R_+$ as in Assumption~\ref{as:lin-growth} such that
  \begin{equation}\label{eq:tonelli}
    \langle b(x) - b(y), x - y \rangle \le f(\|x - y\|)\|x - y\|
  \end{equation}
  for all $x, y \in H$. Then $\xi(x) = \infty$ for all $x $, if $\xi(x_0) = \infty$ for some point $x_0$.

  In cases where the driving noise is a Brownian path, non-explosion and strong completeness can be
  deduced from non-explosion at a single point, as shown in \cite[Theorems 3.1, 4.1, and 6.2]{Li:94a}.
  Here, we employ a condition tailored to accommodate non-differentiable coefficients.

  To prove this claim, define
  $$\tau^k_m = \inf_{\|x - x_0\| \le m} \inf \{t : \|x_t^x - x_t^{x_0}\| \ge Rk\}.$$
  As  $\xi(x_0) = \infty$, we have $\inf_{\|x - x_0\| \le m} \xi(x) \ge \tau^k_m$ for all $k \in \N$.
  It suffices to show that, for some $R > 0$,
  $$\tau^{k + 1}_m - \tau^k_m \ge f(R(k + 1))^{-1}$$
  for all $m \in \N$. Fixing $m$,  write $\tau^k = \tau^k_m$ for simplicity.
  Set $\psi^k_s(x) = x_{\tau^k + s}^x - x_{\tau^k + s}^{x_0}$, we have
  \begin{align*}
    \dd\|\psi^k_s(x)\|^2 & = 2\langle b(x_{\tau^k + s}^x) - b(x_{\tau^k + s}^{x_0}),
    x_{\tau^k + s}^x - x_{\tau^k + s}^{x_0}\rangle\dd s                                                                         \\
                         & \le 2f(\|x_{\tau^k + s}^x - x_{\tau^k + s}^{x_0}\|)\|x_{\tau^k + s}^x - x_{\tau^k + s}^{x_0}\|\dd s.
  \end{align*}
  Take $R > 2 + \sqrt{3}$, then
  $(Rk)^2 + 2R(k + 1) \le (R(k + 1) - 1)^2$
  for all $k \in \N$.  Consequently, for all $s < f(R(k + 1))^{-1}$,
  \begin{align*}
    \|\psi^k_s(x)\|^2 
     & \le (Rk)^2 + 2sf(R(k + 1)) R(k + 1) \le (Rk)^2 + 2R(k + 1).
  \end{align*}
  Thus, $\sup_{0 \le s \le (\tau^{k + 1} - \tau^k) \wedge f(R(k + 1))^{-1}}
    \|x_{\tau^k + s}^x - x_{\tau^k + s}^{x_0}\| \le  R(k + 1) - 1$ and
  $$\sup_{\|x - x_0\| \le m} \left(\sup_{ s \le (\tau^{k + 1} - \tau^k) \wedge f(R(k + 1))^{-1}}
    \|x_{\tau^k + s}^x - x_{\tau^k + s}^{x_0}\| \right)< R(k + 1).$$
  This implies that $\tau^{k + 1} - \tau^k \ge f(R(k + 1))^{-1}$ as desired.
\end{example}

\subsection{Sharpness of Corollary~\ref{cor:global-holder}}

By utilizing a modified version of the ODE from \cite[Lemma 3.15]{cox:24}, we in this section show
that the criterion for non-explosion in Corollary~\ref{cor:global-holder} is sharp in the following sense.

\begin{proposition}[Counterexample]\label{pro:sharp}
  For any $\epsilon > 0$, there exists some $\alpha > 0$, $\gamma \in C^{(1 - \epsilon)\alpha}(\R_+; \R^2)$
  and a vector field $b \in \Lip_{\loc}(\R^2; \R^2)$ satisfying Assumption~\ref{cond:perp} with $\beta = 1 + \alpha$
  such that the solution to the ODE
  $$\dd x_t = b(x_t) \dd t + \dd \gamma_t$$
  explodes in finite time.

  More precisely, we will take $b(x) = \|x\|^{1 + \alpha}Jx$ with $J$ being the $2 \times 2$ matrix corresponding
  to the counter-clockwise rotation by $\pi / 2$ and we define $\gamma$ according to Equation~\eqref{eq:gamma-def}.
\end{proposition}
\begin{proof}
  Setting $\alpha \in (0, 1], \mu \in (\frac{\alpha}{2}, \alpha)$, by taking $\alpha$ sufficiently small,
  it suffices to construct an ODE of the form
  \begin{equation}\label{eq:ode-y}
    x'_t = \|x_t\|^{1 + \alpha}J x_t + \gamma_t'
  \end{equation}
  where $J$ is the counter-clockwise rotation matrix by $\pi / 2$, and $\gamma : \R_+ \to \R^2$ is
  a $\mu / (1 + \alpha)$-Hölder continuous path for which the solution to the ODE~\eqref{eq:ode-y}
  has finite time explosion.

  Let $z : [0, \xi) \to \R^2$ (with $\xi \in (0, \infty]$) be the maximal solution (whose existence
    is guaranteed by Picard-Lindel\"of) to the ODE
    \begin{equation}
      z'_t = \left\|z_t - \|z_t\|^{-(1 + \mu)}Jz_t\right\|^{1 + \alpha} J
      \left(z_t - \|z_t\|^{-(1 + \mu)}Jz_t\right)
    \end{equation}
    with some non-zero initial condition $z_0$. We observe that
    \begin{align*}
      \partial_t \|z_t\|^2 & = 2 \<z_t, z_t'\>
      = 2\left\|z_t - \|z_t\|^{-(1 + \mu)}Jz_t\right\|^{1 + \alpha}\<z_t, J \left(- \|z_t\|^{-(1 + \mu)}Jz_t\right)\> \\
                           & = 2\left\|z_t - \|z_t\|^{-(1 + \mu)}Jz_t\right\|^{1 + \alpha} \|z_t\|^{2 - (1 + \mu)}
      \ge 2 \|z_t\|^{2 + \alpha - \mu}.
    \end{align*}
    Thus, $\|z_t\|$ is increasing and moreover, has a finite blow-up time $\xi < \infty$.
    Consequently, setting
    \begin{equation}\label{eq:gamma-def}
      \gamma_t =
      \begin{cases}
        - \|z_t\|^{-(1 + \mu)}Jz_t, & \text{for } t < \xi, \\
        0,                          & \text{otherwise},
      \end{cases}
    \end{equation}
  $\gamma$ is well-defined and by taking $x_t = z_t + \gamma_t$, we have that $x$ solves the ODE~\eqref{eq:ode-y} up to
    time $\xi$. Moreover, as $\lim_{t \uparrow \xi} \|\gamma_t\| = 0$, it follows that $x$ has the same blow-up time
    as $z$, i.e. $\lim_{t \uparrow \xi} \|x_t\| = \lim_{t \uparrow \xi} \|z_t\| = \infty$.
    Thus, it remains to check that $\gamma_t$ is $\mu / (1 + \alpha)$-H\"older continuous.

    To this end, we apply the quotient rule to $\gamma_t$ for any $t < \xi$ to obtain
    \begin{align*}
      \gamma_t' & = \frac{\|z_t\|^{1 + \mu}(-Jz_t)' - (-Jz_t)(\|z_t\|^{1 + \mu})'}{\|z_t\|^{2(1 + \mu)}}
      = \frac{- Jz_t'}{\|z_t\|^{1 + \mu}} +
      (1 + \mu)\frac{\<z_t, z_t'\> Jz_t}{\|z_t\|^{3 + \mu}}.
    \end{align*}
    Noticing
    \begin{align*}
      \<z_t, z_t'\>
       & = \left\|z_t + \gamma_t\right\|^{1 + \alpha} \frac{\|z_t\|^2}{\|z_t\|^{1 + \mu}}
      \le 2^{\frac{1 + \alpha}{2}}\left(\|z_t\|^{2 + \alpha - \mu} + \|z_t\|^{1 - (2 + \alpha)\mu}\right),
    \end{align*}
    and
    \begin{align*}
      \|Jz_t'\| & = \left\|z_t + \gamma_t\right\|^{2 + \alpha} \le 2^{1 + \frac{\alpha}{2}} (\|z_t\|^{2 + \alpha} + \|z_t\|^{-\mu(2 + \alpha)}),
    \end{align*}
    we have,
    \begin{align*}
      \|\gamma_t'\| \le \
       & + 2^{1 + \frac{\alpha}{2}}\|z_t\|^{1 + \alpha - \mu}
      + 2^{\frac{1 + \alpha}{2}}(\sqrt{2} + 1 + \mu)\|z_t\|^{-1 - (3 + \alpha)\mu}
      + (1 + \mu)2^{\frac{1 + \alpha}{2}}\|z_t\|^{- (2 \mu - \alpha)}.
    \end{align*}
    Since $\|z_t\| \to \infty$ as $t \to \xi$, we may take $t_1 \in [0, \xi)$ such that,
    for all $t \in [t_1, \xi)$, $\|z_t\| \ge 1$. Then, we have $\|\gamma_t'\| \le C\|z_t\|^{1 + \alpha - \mu}$
    for some constant $C > 0$ and any $t \in [t_1, \xi)$. On the other hand, by orthogonality, we have
  $\|\gamma_t'\| \ge \frac{\|Jz'_t\|}{\|z_t\|^{1 + \mu}} \ge \|z_t\|^{1 + \alpha - \mu}$, and
    so $\|\gamma_t'\| \asymp \|z_t\|^{1 + \alpha - \mu}$ for all $t \in [t_1, \xi)$. Thus, using the fact
    that $\gamma_\xi = 0$, there exists some $\delta' \in (0, \delta)$ such that
    $$\delta \|z_{\xi - \delta}\|^{1 + \alpha - \mu}\le \delta \|z_{\xi - \delta'}\|^{1 + \alpha - \mu} \asymp \delta \|\gamma_{\xi - \delta'}'\|
      = \|\gamma_{\xi - \delta} - \gamma_{\xi}\| = \|z_{\xi - \delta}\|^{-\mu}.$$
    Consequently, $\|z_{\xi - \delta}\| \lesssim \delta^{-\frac{1}{1 + \alpha}}$
    for which we obtain $\|\gamma_t'\| \lesssim \delta^{-\left(1 - \frac{\mu}{1 + \alpha}\right)}$ for all $t \in [t_0, \xi)$.
    Hence, by l'H\^opital's rule for limit supremum, we obtain
    \begin{align*}
      \limsup_{\delta \downarrow 0} \delta^{-\frac{\mu}{1 + \alpha}}\|\gamma_{\xi - \delta} - \gamma_{\xi}\|
       & \le \limsup_{\delta \downarrow 0} \left(\frac{\mu}{1 + \alpha} \delta^{1 - \frac{\mu}{1 + \alpha}}\right)
      (\|\gamma_{\xi - \delta}\|)'                                                                                                           \\
       & = \limsup_{\delta \downarrow 0} \left(\frac{\mu}{1 + \alpha} \delta^{1 - \frac{\mu}{1 + \alpha}}\right)
      \<\frac{\gamma_{\xi - \delta}}{\|\gamma_{\xi - \delta}\|}, \gamma_{\xi - \delta}'\>                                                    \\
       & \le \limsup_{\delta \downarrow 0} \left(\frac{\mu}{1 + \alpha} \delta^{1 - \frac{\mu}{1 + \alpha}}\right)\|\gamma_{\xi - \delta}'\|
      < \infty
    \end{align*}
    from which we may conclude that $\gamma$ is $\mu / (1 + \alpha)$-Hölder continuous.
\end{proof}

\section{Non-explosion of RDEs}\label{sec:RDE}
The remainder of this article is dedicated to providing a criterion for non-explosion of solutions 
to RDEs with unbounded coefficients with unbounded derivatives. To this end, we aim to apply 
Lemma~\ref{lem:fast-growth} by viewing the rough integral in the RDE as the driving noise to an ODE. 
Namely, we try to show that the rough integral satisfy the interactive regularity~\eqref{eq:local-holder}. 
As an example, we first illustrate this idea by considering the Young case.

Throughout this section, we let $x_0 \in H$, $b : H \to H$, and $\sigma : H \to \CL(V; H)$ with $b, \sigma$
bounded on bounded sets\footnote{This condition is needed to ensure the existence of maximal solutions
  (cf. \cite[Exercise 8.4]{Friz:20}). This condition is always satisfied in the finite dimensional setting 
  assuming continuity of the coefficients.}.

\subsection{Non-explosion of YDEs}\label{sec:Young}


We consider the YDE
\begin{equation}\label{eq:young-de'}
  x_t = x_0 + \int_0^t b(x_s) \dd s + \int_0^t \sigma(x_s) \dd \gamma_s,
\end{equation}
for which $\gamma : \R_+ \to V$ and the coefficients are sufficiently regular (to be specified) so that the second integral
can be interpreted in terms of Young integration. In this case, the maximal solution is defined as
in Definition~\ref{def:sol} with $\gamma$ replaced by
$\eta_t = \int_0^t \sigma(x_s) \dd \gamma_s$.

We first provide a useful a priori estimate on the H\"older norm of $\eta$.  Below,
$\|\cdot\|_{\alpha; [s, t]}$ denotes the $\alpha$-H\"older norm on the interval $[s, t]$,
$\|\cdot\|_\alpha$ denotes the $\alpha$-H\"older norm on $\R$ or on an implicit interval.

\begin{lemma} \label{lem:young-est}
  Assume that $\gamma \in C_{\loc}^\alpha$ for some $\alpha > \frac{1}{2}$, $\sigma \in C^1$
  and $(x_t) : [0, \xi) \to H$ is a maximal solution to~\eqref{eq:young-de'}.
  Then, for $0 \le s < t < \xi$ with $|t - s|^\alpha \le (2 C_\alpha \|\gamma\|_\alpha\|D\sigma(x_\cdot)\|_{\infty; [s, t]})^{-1}$,
  we have that
  \begin{align*}
    \|\eta\|_{\alpha; [s, t]} \le 2 & \|\sigma(x_\cdot)\|_{\infty; [s, t]} \|\gamma\|_\alpha                                 \\
                                    & + 2 C_{\alpha}\|\gamma\|_\alpha |t - s|^{\alpha} \|D\sigma(x_\cdot)\|_{\infty; [s, t]}
    \left\|\int_0^\cdot b(x_r) \dd r\right\|_{\alpha; [s, t]}
  \end{align*}
  where $C_\alpha = 2^{2\alpha} \zeta(2\alpha)$ with $\zeta$ being the zeta function\footnote{This constant is
    precisely the constant coming from the sewing lemma.}.
\end{lemma}
\begin{proof}
  For simplicity, we drop the interval $[s, t]$ from the notation of the H\"older norm.
  We have by the Young inequality that for all $u < v \in (0, \xi)$, we have
  \begin{align*}
    \|\eta\|_{\alpha}
     & \le \|\sigma(x_\cdot)\|_{\infty} \|\gamma\|_\alpha
    + C_{\alpha}\|\gamma\|_\alpha |t - s|^{\alpha} \|D\sigma(x_\cdot)\|_{\infty; [s, t]} \|x\|_{\alpha}.
  \end{align*}
  Moreover, since
  $\|x\|_\alpha \le \|\eta\|_\alpha + \left\|\int_0^\cdot b(x_r) \dd r\right\|_\alpha$, it follows
  \begin{align*}
    \frac{1}{2} \|\eta\|_{\alpha}
    \le & \left(1 - C_{\alpha}\|\gamma\|_\alpha |t - s|^{\alpha} \|D\sigma(x_\cdot)\|_{\infty; [s, t]}\right)\|\eta\|_{\alpha} \\
    \le & \|\sigma(x_\cdot)\|_{\infty; [s, t]} \|\gamma\|_\alpha
    + C_{\alpha}\|\gamma\|_\alpha |t - s|^{\alpha} \|D\sigma(x_\cdot)\|_{\infty; [s, t]}
    \left\|\int_0^\cdot b(x_r) \dd r\right\|_{\alpha},
  \end{align*}
  for which we obtain the claimed estimate.
\end{proof}

\begin{theorem}\label{thm:young}
  Let $b\in \Lip_{\loc}$, $\sigma \in C^{\frac{1}{\alpha} - 1}$ and $\gamma \in C^\alpha_{\loc}$ for some $\alpha > \frac{1}{2}$.
  Assume that $b$ satisfy the linear growth Assumption~\ref{as:lin-growth} with control $f$, and there exists
  some $\theta, \kappa \in [0, 1)$ such that for all $x \in H$,
  \begin{longlist}
    \item[(Y.1)]\label{cond:young-2} $\|b(x)\| \le f(\|x\|)^{1 + \kappa \alpha},$
    \item[(Y.2)] $\|D^n \sigma(x)\| \le f(\|x\|)^{(\theta - n\kappa) \alpha}$ for $n = 0, 1$.
  \end{longlist}
  Then, for any initial condition, the YDE~\eqref{eq:young-de'} has a global solution, 
  which is furthermore unique if $\sigma \in C^{\frac{1}{\alpha}}$.
\end{theorem}
\begin{proof}
  YDEs with the assumed regularity are known to be locally well-posed 
  (cf. \cite[Theorem 2.1]{Caravenna:25}) and thus, it remains to show that $\xi > T$ for any $T > 0$.

  We denote $\|\gamma\|_\alpha$ for its $\alpha$-H\"older norm on $[0, T]$,
  $\tau^r = \inf\{t : \|x_t\| \ge r\}$ and for all $t \in (0, \xi)$,
  $$\eta_t = \int_0^t \sigma(x_s) \dd \gamma_s.$$
  Since $b$ satisfies Assumption~\ref{cond:perp} for any $\beta > 1$, by choosing $\beta = 1 + (1 - \theta)\alpha$,
  we aim to apply Lemma~\ref{lem:fast-growth} where we regard $(x_t)$ as satisfying the integral
  equation~\eqref{eq:ode-intro} driven by $(\eta_t)$. For this, we seek a pair of numbers $R$
  and $K$ satisfying Equation~\eqref{eq:local-holder} and \eqref{eq:R-bd} uniformly in $k \in \N$.

  Beginning with any $\bar R > 1$, and $k \in \N$, since
  $f(r) \to \infty$ as $r \to \infty$ and $\|D\sigma\|_{\infty; B_{\bar R(k + 1)}}
    \lesssim f(\bar R (k + 1))^{(\theta - \kappa)\alpha}$, there exists some $k_0 \in \N$ such that for all $k \ge k_0$,
  $$f(\bar R (k + 1))^{-\alpha} \le (2 C_\alpha \|\gamma\|_\alpha \|D\sigma\|_{\infty; B_{\bar R(k + 1)}})^{-1}.$$
  Thus, by applying Lemma~\ref{lem:young-est} to the interval
  $$I = I(k, \bar R) = [\tau^{\bar Rk}, \min\{\tau^{\bar Rk} + f(\bar R (k + 1))^{-1}, \tau^{\bar R(k + 1)}\}],$$
  we have for all $k \ge k_0$ that
  \begin{align*}
    \|\eta\|_{\alpha; I}
    \le & \ 2\|\sigma\|_{\infty; B_{\bar R(k + 1)}} \|\gamma\|_\alpha                                                                      \\
        & + 2 C_{\alpha}\|\gamma\|_\alpha  \|D\sigma\|_{\infty; B_{\bar R(k + 1)}} f(\bar R (k + 1))^{-1}\|b\|_{\infty; B_{\bar R(k + 1)}} \\
    \le & \ 2 f(\bar R (k + 1))^{\theta \alpha}\|\gamma\|_\alpha                                                                           \\
        & + 2 C_{\alpha}\|\gamma\|_\alpha f(\bar R (k + 1))^{(\theta - \kappa)\alpha}
    f(\bar R (k + 1))^{-1} f(\bar R(k + 1))^{1 + \kappa\alpha}                                                                             \\
    \le & (2\|\gamma\|_\alpha + 2 C_{\alpha}\|\gamma\|_\alpha)f(\bar R(k + 1))^{\theta \alpha}.
  \end{align*}
  We have that,
  for all $t \le f(\bar R (k + 1))^{-1} \wedge (\tau^{\bar R(k + 1)} - \tau^{\bar Rk})$,
  \begin{align*}
    \sup_{0 \le s \le t}\|\eta_{\tau^{\bar R k} + s} - \eta_{\tau^{\bar R k}}\|
     & \le(2\|\gamma\|_\alpha + 2 C_{\alpha}\|\gamma\|_\alpha)f(\bar R(k + 1))^{\theta \alpha} t^{\alpha} \\
     & \le 2(1+C_\alpha) \|\gamma\|_\alpha \; t^{(1 - \theta) \alpha}.
  \end{align*}
  Consequently, choosing $R > \|x_0\|$ such that Equation~\eqref{eq:R-bd} holds with  $K=2(1+C_\alpha) \|\gamma\|_\alpha$.
  In particular, \eqref{cond:gamma-reg} holds uniformly, i.e.
  $$\sup_{0 \le s \le t}\|\eta_{\tau^{R k} + s} - \eta_{\tau^{ R k}}\|\le K \delta_k^{(1 - \theta)\alpha},$$
  for all $k \ge k_0$ and $t \le \delta_k = \min\{1, f(R(k + 1))^{-1}, (a / K)^{\frac{1}{\beta - 1}}\}$.
  Thus, $\xi > T$ follows from Lemma~\ref{lem:fast-growth}, as required.
\end{proof}

\subsection{Preliminaries on rough path theory}\label{sec:rp-pre}

We recall some notions from controlled rough path theory used in the remainder of this article.
\begin{definition}[Rough paths]
  For $\alpha \in (\frac{1}{3}, \frac{1}{2}]$, an $\alpha$-rough path taking values in $V$ is
  a pair $\X = (X, \XX)$ with $X \in C^{\alpha}([0, T]; V)$ and $\XX \in \C^{2\alpha}(\Delta_T; V \otimes V)$
  satisfying Chen's relation:
  \begin{equation}\label{eq:chen}
    \delta \XX_{s, u, t} = \XX_{s,t}-\XX_{s,u}-\XX_{u,t}=X_{s,u}\otimes X_{u,t},
  \end{equation}
  for any $s\leq u\leq t$.

  We denote by $\C^\alpha([0, T]; V)$ the space of all $\alpha$-rough paths and equip it with the norm
  $\|\X\|_{\alpha} = \|X\|_{\alpha} + \|\XX\|_{2 \alpha}$. Moreover, we write $\X \in \C^{\alpha+}([0, T]; V)$
  if there exists some $\beta > \alpha$ such that $\X \in \C^{\beta}([0, T]; V)$.
\end{definition}

\begin{definition}[Controlled rough paths]
  Let $X \in C^\alpha([0, T]; V)$, we say $Y \in C^\alpha([0, T]; W)$
  is controlled by $X$ if there exists a $Y' \in\C^\alpha([0, T]; \CL(V; W))$ such that defining
  \begin{equation}
    R_{s, t}^Y = Y_{s, t} - Y'_s X_{s, t}
  \end{equation}
  for any $(s, t) \in \Delta_T$, we have that $R \in \C^{2\alpha}(\Delta_T; W)$. We call $Y'$
  the Gubinelli derivative of $Y$ against $X$ albeit it might not be unique.

  The pair $(Y, Y')$ is known as a controlled rough path and we denote the space of all such controlled
  rough path by $\D_X^{2 \alpha}([0, T]; W)$.
\end{definition}

Now, for $\alpha \in (\frac{1}{3}, \frac{1}{2}]$, $\X = (X, \XX) \in \C^\alpha([0, T]; V)$,
$(Y, Y') \in \D^{2\alpha}_X([0, T]; W)$, the rough integral of $(Y, Y')$ against $\X$ is
defined by the Riemann sum approximation
$$\int_s^t Y_r \dd \X_r :=
  \sum_{(t_i, t_{i + 1}) \in \mathcal{P}} (Y_{t_i} X_{t_i, t_{i + 1}} + Y'_{t_i}\XX_{t_i, t_{i + 1}}) + O(|\mathcal{P}|),$$
where $\mathcal P$ is a partition of $[s, t]$ for any $s < t \in [0, T]$. Moreover, we have the estimate
\begin{equation}\label{eq:Remainder-ineq}
  \begin{split}
    & \left\|\int_s^t Y_r \dd \X_r - Y_s X_{s, t} - Y'_s \mathbb{X}_{s, t}\right\|\\
    & \le C_\alpha (\|X\|_\alpha \|R^Y\|_{2\alpha} +
    \|\mathbb{X}\|_{2\alpha} \|Y'\|_{\alpha})|t - s|^{3\alpha}.
  \end{split}
\end{equation}
Strictly speaking, this integral should be written
as $\int_0^t (Y, Y')_s \dd \X_s$ since $Y'$ contributes towards the limit. Nonetheless, the choice
of $Y'$ is usually clear from the context and we omit it from the notation.

\begin{definition}
  Let $\alpha \in (\frac{1}{3}, \frac{1}{2}]$, $\X \in \C^{\alpha}([0, T]; V)$,
  $b : W \to W$ and $\sigma \in C^2_b(W; \CL(V; W))$. A maximal solution to the RDE
  \begin{equation}\label{eq:rde}
    \dd x_t = b(x_t) \dd t + \sigma(x_t) \dd \X_t
  \end{equation}
  with initial condition $x_0 \in W$ and life time $\xi \in (0, T]$ is an element
  $(x, x') \in \D^{2\alpha}_X([0, \xi); W)$ with $x' = \sigma(x)$ such that if $\xi < \infty$, then
  $\limsup_{t \to \xi}\|x_t\| = \infty$. Moreover, for all $t \in [0, \xi)$,
  \begin{equation}
    x_t = x_0 + \int_0^t b(x_s) \dd s + \int_0^t \sigma(x_s) \dd \X_s.
  \end{equation}
  Here, the latter integral is understood as a rough integral with
  $\sigma(x_\cdot)$ having the Gubinelli derivative $D\sigma(x_\cdot) x_\cdot' = D\sigma(x_\cdot) \sigma(x_\cdot)$
  against $\X$ (cf. \cite[Section 7.3]{Friz:20}).

  In the case where we can take $\xi = \infty$, we say that the solution is global.
\end{definition}

\subsection{Uniform bound for the H\"older norm on the interval}

In contrast to the bound of the Young integral in which we utilized the linear nature of the Young
inequality, the remainder inequality for the rough integral will result in a quadratic term. While
such an inequality could provide a bound if we know a priori that the quantity is finite, a bound obtained
this way will not be uniform along the different level sets which is needed to show the interactive regularity
required by Lemma~\ref{lem:fast-growth}. This motivates the following lemma.

\begin{lemma}[Flying fish]\label{lem:poly-bound}
  Let $a, b, c\in C(\R_+; \R)$ be such that $\liminf_{x \to \infty} a(x) > 0$, $a(0) < 0$
  and $c$ is non-decreasing with $c(0) = 0$. Then, denoting
  $$x^* = \sup \{x \in \R_+ : a(x) = 0\}$$
  and
  \begin{equation}\label{eq:ineq-class}
    \mathcal{A} =
    \left\{f \in C([0, 1]; \R_+) \;\middle|\;
    \begin{aligned}
       & f \text{ is non-decreasing}                                                 \\
       & \forall \epsilon \in (0, 1),\ a(f(\epsilon)) \le c(\epsilon) b(f(\epsilon))
    \end{aligned}
    \right\},
  \end{equation}
  for any $r > 0$, there exists some $\epsilon^* > 0$ such that
  $\sup_{f \in \mathcal{A}} f(\epsilon^*) \le x^* + r$.
\end{lemma}
Before presenting the proof, let us first provide some graphical intuition for why we should expect such 
a uniform estimate. 
Suppose that $f \in \mathcal{A}$ so that $a(f(\epsilon)) - \epsilon b(f(\epsilon)) \le 0$ for any $\epsilon \in (0, 1)$. 
Thus, fixing $\bar \epsilon > 0$, $f(\bar \epsilon)$ must take value at points for which the function 
$x \mapsto a(x) - {\bar \epsilon} b(x)$ is non-positive, i.e. it lives in one of the blue 
ponds in the first figure\footnote{Here $a(x) = x - 5$ and $b(x) = x^2 \sin(x)$} below. 
Now, shrinking $\bar \epsilon$, as $a(x) - {\bar \epsilon} b(x) \to a(x)$ for $\bar \epsilon \to 0$,  
we observe that the left edge of ponds in $[x^*, \infty)$ will move to the right. 
On the other hand, as $f$ is non-decreasing, it will move to the left.
Consequently, if $f(\bar \epsilon)$ lives in a pond in $[x^*, \infty)$, shrinking $\epsilon$ will 
force it to hit the left edge at which point it must jump to a previous pond. This contradicts the 
continuity of $f$.
\begin{center}
  \begin{tikzpicture}
    \begin{axis}[
        legend style={fill=white,draw=black},
        axis x line = middle,
        axis y line = left,
        xmin=0,
        xmax=22, 
        ymin=-30,
        ymax=50,
        xlabel = $x$,
        ylabel = {$a(x)- {\bar \epsilon} b(x)$},
    ]
    \addplot [
        domain=0:22, 
        samples=1000, 
        color=red,
    ]
    {x - 5 - (x^2 * sin(deg(x)) / 10)};
    \addlegendentry{${\bar\epsilon} = 0.1$}
    \addplot [
        domain=0:3.9204, 
        samples=2, 
        color=blue,
        style={ultra thick},
    ]
    {0};
    \addplot [
        domain=6.66771:8.9098, 
        samples=2, 
        color=blue,
        style={ultra thick},
    ]
    {0};
    \addplot [
        domain=13.05858:15.25156, 
        samples=2, 
        color=blue,
        style={ultra thick},
    ]
    {0};
    \addplot [
        domain=19.24436:21.62772, 
        samples=2, 
        color=blue,
        style={ultra thick},
    ]
    {0};
  \end{axis}
  \end{tikzpicture}
  \begin{tikzpicture}
    \begin{axis}[
        legend style={fill=white,draw=black},
        axis x line = middle,
        axis y line = left,
        xlabel = $x$,
        xmin=0,
        xmax=22, 
        ymin=-30,
        ymax=50,
    ]
    \addplot [
        domain=0:22, 
        samples=1000, 
        color=red,
    ]
    {x - 5 - 0.04 * (x^2 * sin(deg(x)))};
    \addlegendentry{${\bar\epsilon} = 0.04$}
    \addplot [
        domain=0:4.31392, 
        samples=2, 
        color=blue,
        style={ultra thick},
    ]
    {0};
    \addplot [
        domain=20.05962:20.84378, 
        samples=2, 
        color=blue,
        style={ultra thick},
    ]
    {0};
  \end{axis}
  \end{tikzpicture}
\end{center}
\begin{proof}
  Define $Q^\epsilon(x) := a(x) - c(\epsilon) b(x)$. For any $r > 0$, by the definition of $x^*$,
  there exists some $\epsilon^* \in (0, 1)$ such that $Q^{\epsilon'}(x^* + r) > 0$ for all $\epsilon \le \epsilon^*$.
  Suppose for contradiction that there exists $f \in \mathcal{A}$ such that $f(\epsilon^*) > x^* + r$.

  Since $\inf_{x \in [x^* + r,f(\epsilon^*)]}a(x) > 0$, we can take $0 < \epsilon' < \epsilon^*$ such that
  $$c(\epsilon')\le \f{\inf_{x \in [x^* + r,f(\epsilon^*)]}a(x)}{2(1 + \|b\|_{\infty; [x^* + r,f(\epsilon^*)]})}.$$
  Then $Q^{\epsilon'}(x) > 0$ for all $x \in [x^* + r, f(\epsilon^*)]$. Since $f\in \mathcal{A}$,
  $Q^{\epsilon'}(f(\epsilon')) \le 0$, therefore $f(\epsilon')\not \in [x^* + r, f(\epsilon^*)]$. By monotonicity,
  $f(\epsilon') \le f(\epsilon^*)$ and it follows that $f(\epsilon') < x^* + r$. Thus,
  $f(\epsilon') < x^* + r < f(\epsilon^*)$ and by intermediate value theorem, there
  exists some $\epsilon \in (\epsilon', \epsilon^*)$ such
  that $f(\epsilon) = x^* + r$. However, by the choice of $\epsilon^*$,
  $Q^\epsilon(f(\epsilon)) = Q^\epsilon(x^* + r) > 0$ which contradicts the fact that
  $Q^{\epsilon}(f(\epsilon)) \le 0$.
\end{proof}

We will want to apply the above lemma to the map $\epsilon \mapsto \|A\|_{\alpha; [t_0, t_0 + \epsilon]}$
for some $A \in \C^\alpha(\Delta_T; B)$. However, it is a priori not known whether or not this
map is continuous. Indeed, this is in general not the case by considering the Lipschitz norm of
the function $(s, t) \mapsto x_t - x_s$, where
$$x_t = t \mathbb{1}_{[0, 1]}(t) + \left(2t - 1\right)\mathbb{1}_{(1, \infty)}(t).$$
Nonetheless, for this particular map, for $\alpha \in (0, 1)$, we observe that
$$\|\delta x\|_{\alpha; [0, \epsilon]} = \|x\|_{\alpha; [0, \epsilon]} =
  \begin{cases}
    \epsilon^{1 - \alpha},        & \text{if } \epsilon \le 1, \\
    \max\left\{2(\epsilon - 1)^{1 - \alpha}, \epsilon^{1 - \alpha} + \epsilon^\alpha -
    \epsilon^{\alpha - 1}\right\} & \text{if } \epsilon > 1,
  \end{cases}$$
which is continuous in $\epsilon$. This turns out to be true in general as demonstrated in the
following lemma.

\begin{lemma}[Continuity of the H\"older norm]\label{lem:cont-holder}
  Let $V$ be a Banach space and $A \in \C^{\alpha}(\Delta_T; V)$ be a two parameter process
  which is jointly continuous and vanishes on its diagonal. Then, for any $t_0 \in [0, T)$ and
  $\alpha' \in (0, \alpha)$, the function
  \begin{align*}
    \epsilon \mapsto \|A\|_{\alpha'; [t_0, t_0 + \epsilon]}
  \end{align*}
  is continuous on $[0, T - t_0]$.
\end{lemma}
A proof is given in  Appendix~\ref{sec:cont-holder} for completeness.

We remark the requirement that $A$ is jointly continuous is not redundant as a two parameter process
in $\C^{\alpha}(\Delta_T; B)$ is not necessarily continuous as a map. This is illustrated by the example
$A : (s, t) \mapsto |t - s|\mathbb{1}_{[1, \infty)}(t)$.

\subsection{Estimates for solutions to RDEs}\label{sec:rough-ests}

We take $\alpha \in (\frac{1}{3}, \frac{1}{2})$, $\mathbb{\Gamma} = (\gamma, \Gamma) \in \C^{\alpha +}(\R_+; V)$,
and $\sigma \in C^{\frac{1}{\alpha}}(H; \CL(V; H))$.
Let $(x, x') \in \mathcal{D}_\gamma^{2\alpha}([0, \xi), H)$ denote the maximal solution to
\begin{equation}\label{eq:rde'}
  \dd x_t = b(x_t) \dd t + \sigma(x_t) \dd \mathbb{\Gamma}_t
\end{equation}
with life time $\xi$. Define $D_\gamma \sigma(x)_t = (D \sigma(x_t)) \sigma(x_t)$
which is a Gubinelli derivative of $\sigma(x)$ against $\gamma$, and we set
\begin{equation}\label{eq:eta-def}
  \eta_t = \int_0^t \sigma(x_s) \dd \mathbb{\Gamma}_s, \qquad R^\eta_{s, t} = \eta_{s, t} - \sigma(x_s) \gamma_{s, t},
\end{equation}
where the integral w.r.t. $\mathbb{\Gamma}$ denotes a second order rough integral.
We recall the following useful estimate.
\begin{lemma}\cite[Lemma 7.3]{Friz:20}
  Fixing $s < t \in [0, \xi)$.
  For any $u < v \in [s, t]$, we denote $R^x_{u, v} = x_{u, v} - \sigma(x_u)\gamma_{u, v}$
  and $R^\sigma_{u, v} = R^{\sigma(x)}_{u, v} = \sigma(x)_{u, v} - D_\gamma \sigma(x)_u \gamma_{u, v}$. Then
  \begin{equation}\label{eq:R-sigma-bdd}
    \begin{split}
      \|R^\sigma\|_{2\alpha; [s, t]}
      \le & \sup_{u < v \in [s, t]} \sup_{\lambda \in [0, 1]} \frac{1}{2}\|D^2\sigma ((1 - \lambda)x_u + \lambda x_v)\|
      \|x\|^2_{\alpha; [s, t]}                                                                                          \\
          & + \sup_{u \in [s, t]}\|D\sigma(x_u)\| \|R^x\|_{2\alpha; [s, t]}.
    \end{split}
  \end{equation}
\end{lemma}

\begin{lemma}\label{lem:R-eta-ineq}
  Fixing $s < t \in [0, \xi)$, we introduce the shorthand $B_t^x = \int_0^t b(x_r) \dd r$,
  \begin{itemize}
    \item $a_n = \|D^n \sigma(x_\cdot)\|_{\infty; [s, t]}$ where $n = 0, 1, 2$,
    \item $a_2' = \sup_{u < v \in [s, t]} \sup_{\lambda \in [0, 1]}\|D^2\sigma ((1 - \lambda)x_u + \lambda x_v)\|$.
  \end{itemize}
  Then, we have the estimate
  \begin{equation}\label{eq:rough-apriori}
    \begin{split}
      \|R^\eta\|_{2\alpha; [s, t]} & \lesssim a_2' \|R^\eta\|_{2\alpha; [s, t]}^2|t - s|^{3\alpha}                         \\
                                   & + (a_1 + (a_1^2 + a_2 a_0)|t - s|^\alpha)\|R^\eta\|_{2\alpha; [s, t]}|t - s|^{\alpha} \\
                                   & + a_2'\|B^x\|_{2\alpha; [s, t]}^2|t - s|^{3\alpha}                                    \\
                                   & + (a_1 + (a_1^2 + a_2 a_0)|t - s|^\alpha)\|B^x\|_{2\alpha; [s, t]}|t - s|^{\alpha}    \\
                                   & + (a_2' a_0^2 + a_0 a_1^2 + a_2 a_0^2)|t - s|^\alpha + a_0 a_1
    \end{split}
  \end{equation}
  where the positive constant indicated by $\lesssim$ depends on
  $\|\mathbb{\Gamma}\|_{\alpha}$.
\end{lemma}
\begin{proof}
  We omit references to the interval from notations for brevity. Firstly, we observe that
  $$\|R_{s, t}^\eta\| \le \|R^\eta_{s, t} - D_\gamma \sigma(x_s) \Gamma_{s, t}\| + \|D_\gamma \sigma(x_s) \Gamma_{s, t}\|$$
  where straightaway, the second term can be estimated by
  \begin{equation}\label{eq:est-Dgamma}
    \|D_\gamma \sigma(x_s) \Gamma_{s, t}\| \le \|D\sigma(x_s)\sigma(x_s)\| \|\Gamma\|_{2\alpha}|t - s|^{2\alpha}
    \lesssim a_1 a_0|t - s|^{2\alpha}.
  \end{equation}
  For the first term, by writing the remainder estimate~\eqref{eq:Remainder-ineq} in the form
  \begin{equation}\label{eq:R-eta-bdd}
    \begin{split}
      & \|R^\eta_{s, t} - D_\gamma \sigma(x_s) \Gamma_{s, t}\|\\
      & \le C_\alpha (\|\gamma\|_\alpha \|R^{\sigma}\|_{2\alpha; [s, t]} +
      \|\Gamma\|_{2\alpha} \|D_\gamma \sigma(x)\|_{\alpha; [s, t]})|t - s|^{3\alpha},
    \end{split}
  \end{equation}
  it suffices to estimate $\|R^\sigma\|_{2\alpha}$ and $\|D_\gamma\sigma(x)\|_{\alpha}$.
  By~\eqref{eq:R-sigma-bdd}, we have
  \begin{align*}
    \|R^\sigma\|_{2\alpha}
    \le & \f 12 a_2' \|x\|^2_{\alpha} + a_1 \|R^x\|_{2\alpha}\\
    \le & a_2' \|R^x\|^2_{2\alpha} |t - s|^{2\alpha} + a_2' a_0^2 \|\gamma\|_\alpha^2 + a_1 \|R^x\|_{2\alpha}
  \end{align*}
  where in the last step, we used the trivial estimate
  $$\|x\|_{\alpha} \le \|R^x\|_{2\alpha}|t - s|^{\alpha} + \|\sigma(x_\cdot)\|_{\infty; [s, t]} \|\gamma\|_{\alpha}.$$
  Since $R_{s,t}^x=R^\eta_{s,t}+B_{s,t}^x$, we also have
  $$\|R^x\|_{2\alpha} \le \|R^\eta\|_{2\alpha} + \|B^x\|_{2\alpha}.$$
  Hence, combining the above, we have
  \begin{equation}\label{eq:R-sigma-bdd'}
    \begin{split}
      \|R^\sigma\|_{2\alpha}
      \le & 2 a_2' (\|R^\eta\|_{2\alpha}^2 + \|B^x\|_{2\alpha}^2) |t - s|^{2\alpha}\\
      & + a_2' a_0^2 \|\gamma\|_\alpha^2 + a_1 (\|R^\eta\|_{2\alpha} + \|B^x\|_{2\alpha}).
    \end{split}
  \end{equation}
  On the other hand, we observe
  \begin{equation}\label{eq:deriv-gamma-bdd}
    \begin{split}
      \|D_\gamma\sigma(x)\|_{\alpha}
      \le & (a_1^2 + a_2 a_0)\|x\|_\alpha\\
      \lesssim & (a_1^2 + a_2 a_0)(\|R^x\|_{2\alpha}|t - s|^\alpha + a_0 \|\gamma\|_\alpha)\\
      \lesssim & (a_1^2 + a_2 a_0)(\|R^\eta\|_{2\alpha} + \|B^x\|_{2\alpha})|t - s|^\alpha + a_0(a_1^2 + a_2 a_0).
    \end{split}
  \end{equation}
  Thus, substituting \eqref{eq:R-sigma-bdd'} and \eqref{eq:deriv-gamma-bdd} into \eqref{eq:R-eta-bdd},
  we obtain
  \begin{align*}
             & \|R^\eta_{s, t} - D_\gamma \sigma(x_s) \Gamma_{s, t}\|                                               \\
    \lesssim & a_2' (\|R^\eta\|^2_{2\alpha} + \|B^x\|^2_{2\alpha}) |t - s|^{5\alpha}                                \\
             & + (a_1 + (a_1^2 + a_2 a_0)|t - s|^\alpha)(\|R^\eta\|_{2\alpha} + \|B^x\|_{2\alpha})|t - s|^{3\alpha} \\
             & + (a_2' a_0^2 + a_0 a_1^2 + a_2 a_0^2) |t - s|^{3\alpha}.
  \end{align*}
  This combined with \eqref{eq:est-Dgamma} completes the proof.
\end{proof}

\begin{assumption}\label{assume2}
  Let $b\in \Lip_{\loc}(H; H)$ satisfies Assumption~\ref{as:lin-growth} with control $f$. 
  Let $\mathbb{\Gamma} = (\gamma, \Gamma) \in \C^{\alpha+}(\R_+; V)$ for some 
  $\alpha \in (\frac{1}{3}, \frac{1}{2})$ and $\sigma \in C^{\frac{1}{\alpha}}(H; \CL(V; H))$.
  Furthermore, suppose that for all $x \in H$,
  \begin{longlist}
    \item[(R.1)]\label{eq:kappa_b} $\|b(x)\| \le f(\|x\|)^{1 + \kappa \alpha}$,
    \item[(R.2)]\label{eq:kappa_n} $\|D^n \sigma(x)\| \le f(\|x\|)^{(\theta - n \kappa) \alpha}$ for $n = 0, 1, 2$,
  \end{longlist}
  for some $\theta \in [0, 1)$ and $\kappa \in [0, \f 1 2)$.
\end{assumption}

We remark that by possibly increasing $\theta$, we can always assume $2 \kappa < \theta$ so that
the exponents in \ref{eq:kappa_n} is positive. Moreover, the assumption
$\sigma \in C^{\frac{1}{\alpha}}(H; \CL(V; H))$ is assumed to obtain an a priori existence
of a unique maximal solution. As our strategy is to show that any maximal solutions are in fact global, 
for only existence of a global solution, it will be instead sufficient to assume 
$\sigma \in C^{\frac{1}{\alpha} - 1}(H; \CL(V; H))$ (cf. \cite[Theorem 3.2, 3.3]{Davie:07}, 
\cite[Theorem 1]{Lejay:09}). 

\begin{lemma}\label{lem:R-bdd}
  Suppose Assumption~\ref{assume2} holds with constants $\theta$ and $\kappa$. Then,
  there exists some $K \ge 0$, such that for any $R \ge 0$ and any sufficiently large $k \in \N$,
  $$(\epsilon_k)^{(1 + \theta)\alpha} \|R^\eta\|_{2\alpha; [\tau^{Rk}, (\tau^{Rk} + \epsilon_k) \wedge \tau^{R(k + 1)}]} \le K$$
  where $\epsilon_k = f(R(k + 1))^{-1}$ and $\tau^r = \inf\{t : \|x_t\| \ge r\}$.
\end{lemma}
\begin{proof}
  Denoting $I(k, \epsilon) = [\tau^{Rk}, (\tau^{Rk} + \epsilon) \wedge \tau^{R(k + 1)}]$, we define the map
  $$N^k : \epsilon \mapsto N^k_\epsilon := \epsilon^{(1 + \theta)\alpha} \|R^\eta\|_{2\alpha; I(k, \epsilon)},$$
  which is continuous and non-decreasing by Lemma~\ref{lem:cont-holder}.

  We aim to apply Lemma~\ref{lem:poly-bound}. Adopting the notations as in Lemma~\ref{lem:R-eta-ineq},
  by multiplying both sides of inequality~\eqref{eq:rough-apriori} by $\epsilon^{(1 + \theta)\alpha}$,
  we have
  \begin{equation}\label{eq:est-Nk}
    \begin{split}
      N^k_\epsilon & \lesssim a_2' (N^k_\epsilon)^2 \epsilon^{(2 - \theta)\alpha}
      + (a_1 + (a_1^2 + a_2 a_0) \epsilon^\alpha)N^k_\epsilon \epsilon^{\alpha}\\
      & + a_2'\|B^x\|_{2\alpha; I(k, \epsilon)}^2 \epsilon^{(4 + \theta)\alpha}
      + (a_1 + (a_1^2 + a_2 a_0) \epsilon^\alpha)\|B^x\|_{2\alpha; I(k, \epsilon)}\epsilon^{(2 + \theta)\alpha}\\
      & + (a_2' a_0^2 + a_0 a_1^2 + a_2 a_0^2) \epsilon^{(2 + \theta)\alpha}
      + a_0 a_1\epsilon^{(1 + \theta)\alpha}.
    \end{split}
  \end{equation}
  Now, by Assumption \ref{eq:kappa_b} on $b$, for $\epsilon \le \epsilon_k$, we observe that
  \begin{align*}
    \|B^x\|_{2\alpha; I(k, \epsilon)} & \le \sup_{s < t \in I(k, \epsilon)} \frac{1}{|t - s|^{2\alpha}} \int_s^t \|b(x_r)\|\dd r \\
                                      & \le \epsilon_k^{1 - 2\alpha} f(R(k + 1))^{1 + \kappa \alpha}
    = \epsilon_k^{-(2 + \kappa) \alpha}.
  \end{align*}
  Hence, this estimate alongside the Assumption~\ref{eq:kappa_n} on $\sigma$ implies
  $$a_2'\|B^x\|_{2\alpha; I(k, \epsilon)}^2 \epsilon^{(4 + \theta)\alpha}
    \lesssim \epsilon^{(4 + \theta - \theta + 2\kappa - 2(2 + \kappa))\alpha} = 1.$$
  Denoting $\lambda = (1 - \theta + \kappa)\alpha$, we have
  $$(a_1 + (a_1^2 + a_2 a_0) \epsilon^\alpha) \epsilon^\alpha
    \lesssim \epsilon^{((1 - \theta + \kappa) \wedge (2 - 2\theta + 2\kappa))\alpha} = \epsilon^{\lambda},$$
  and so
  $$(a_1 + (a_1^2 + a_2 a_0) \epsilon^\alpha)\|B^x\|_{2\alpha; I(k, \epsilon)}\epsilon^{(2 + \theta)\alpha}
    \lesssim \epsilon^{(1 + \theta + (1 - \theta + \kappa) - (2 + \kappa))\alpha} = 1.$$
  Moreover, $(a_2' a_0^2 + a_0 a_1^2 + a_2 a_0^2)\epsilon^{(2 + \theta)\alpha} \lesssim \epsilon^{2\lambda}$.
  Thus, by substituting the above estimates into \eqref{eq:est-Nk}, we obtain that, for any $\epsilon \le \epsilon_k$
  \begin{align*}
    N^k_\epsilon \lesssim & (N^k_\epsilon)^2 \epsilon^{2\lambda}
    + N^k_\epsilon (\epsilon^{\lambda} + \epsilon^{2\lambda} )+ 1 + \epsilon^{\lambda} + \epsilon^{2\lambda}.
  \end{align*}
  Namely, for any $\epsilon \in (0, \epsilon_k \wedge 1]$, we have
  $$N^k_\epsilon \le c\epsilon^{\lambda}((N^k_\epsilon)^2 + N^k_\epsilon) + c,$$
  for some constant $c>0$ independent of $k$ and $R$.
  Thus, setting
  $$\tilde N^k_\epsilon = N^k_\epsilon \mathbb{1}_{\epsilon \le \epsilon_k} + N^k_{\epsilon_k}\mathbb{1}_{\epsilon > \epsilon_k}$$
  for $\epsilon_k < 1$, we have that
  $$\tilde N^k_\epsilon - c \le c \epsilon^{\lambda}((\tilde N^k_\epsilon)^2 + \tilde N^k_\epsilon)$$
  for all $\epsilon \in (0, 1)$. Consequently, we may apply Lemma~\ref{lem:poly-bound} to
  $\{\tilde N^k\}_{k \in \N}$, from which we obtain some $\epsilon^* > 0$ and $K \in \R_+$
  such that $\tilde N^k_{\epsilon^*} \le K$ for all $k$.
  Finally, since $\epsilon_k = f(R(k + 1))^{-1} \downarrow 0$ as $k \to \infty$, it follows that
  $\epsilon_k \le \epsilon^*$ for all sufficiently large $k$ and thus for such $k$,
  $N_{\epsilon_k}^k = \tilde N_{\epsilon_k}^k \le \tilde N^k_{\epsilon^*} \le K$ as claimed.
\end{proof}

\begin{theorem}\label{thm:rough}
  Under Assumption \ref{assume2}, the RDE~\eqref{eq:rde'} has a unique global solution.
\end{theorem}
\begin{proof}
  Under the assumed regularities, it is known \cite{Davie:07, Lejay:09, Lejay:12} that the RDE has a unique maximal solution
  $(x, \sigma(x)) \in \mathcal{D}^{2\alpha}_\gamma([0, \xi), H)$
  and it remains to show $\xi = \infty$.

  Since Assumption~\ref{cond:perp} holds with any $\beta > 1$, we aim to apply Lemma~\ref{lem:fast-growth},
  where we take $\gamma$ to be $\eta$ from \eqref{eq:eta-def} and  choose $\beta = 1 + (1 - \theta)\alpha$.

  For any $\bar R > 0$ and $k \in \N$, let us denote $\epsilon_{\bar R(k + 1)} = f(\bar R(k + 1))^{-1}$ and
  $I(k, \bar R) = [\tau^{\bar Rk}, (\tau^{Rk} + \epsilon_{R(k + 1)}) \wedge \tau^{R(k + 1)}]$.
  By Lemma~\ref{lem:R-bdd}, there exists some $\tilde K > 0$ such that for
  any $\bar R > 0$, there exists some $k_0 \in \N$ for which for any $k \ge k_0$ and $s < t \in I(k, \bar R)$,
  $$(\epsilon_{\bar R(k + 1)})^{(1 + \theta)\alpha}\|R^\eta\|_{2\alpha; I(k, \bar R)}\le \tilde K.$$
  Hence,
  \begin{align*}
    \|\eta_{s, t}\| & = \|R^\eta_{s, t} + \sigma(x_s)\gamma_{s, t}\|                                                                        \\
                    & \le \|R^\eta\|_{2\alpha; I(k, \bar R)} (\epsilon_{\bar R(k + 1)})^{(1 + \theta)\alpha}|t - s|^{(1 - \theta) \alpha} +
    \|\sigma(x_\cdot)\|_{\infty; I} \|\gamma\|_{\alpha}|t - s|^\alpha                                                                       \\
                    & \lesssim \left(\tilde K + f(\bar R(k + 1))^{\theta\alpha}
    (\epsilon_{\bar R(k + 1)})^{\theta\alpha}\right)|t - s|^{(1 - \theta) \alpha}                                                           \\
                    & = (\tilde K + 1) |t - s|^{(1 - \theta) \alpha}.
  \end{align*}
  Thus, taking $R > 0$ sufficiently large such that the pair $(R, K)$ where $K=\tilde K+1$ satisfies Equation~\eqref{eq:R-bd},
  we have that, for all $k \ge k_0$,
  $$\sup_{0 \le s \le \sigma_k \wedge \delta_k} \|\eta_{\tau^{R k} + s} - \eta_{\tau^{R k}}\|
    \le K \delta_k^{(1 - \theta) \alpha} = K \delta_k^{\beta - 1}$$
  where $\delta_k = \min\{1, f(R(k + 1))^{-1}, (a / K)^{\frac{1}{\beta - 1}}\}$.
  This is precisely the assumption required by Lemma~\ref{lem:fast-growth} allowing us to conclude non-explosion
  of the solution.
\end{proof}

By utilizing the continuity theorem for RDE flows, the above theorem will allow us to establish the
strong completeness of SDE~\eqref{eq:sde-rde} with multiplicative noise.

Let $W$ be a Brownian motion on $\R^l$ and we define $\mathbf{W}$ as a rough path via the It\^o lift
$$\mathbf{W}^{\text{It\^o}} = (W, \mathbb{W}^{\text{It\^o}}) \in \C^{\frac{1}{2}-} \text{ where }
  \mathbb{W}^{\text{It\^o}}_{s, t} = \int_s^t W_{s, r} \otimes \dd W_r$$
which is defined up to a set of full measure $\Omega_B$.

\begin{corollary}\label{cor:strong-complete}
  Under Assumption~\ref{assume2} with $\alpha < \frac{1}{2}$, the SDE
  \begin{equation}\label{eq:sde-rde}
    \dd x_t = b(x_t) \dd t + \sigma(x_t) \dd W_t
  \end{equation}
  is strongly complete.
\end{corollary}
\begin{proof}
  Since the coefficients of the SDE are sufficiently regular, it is well known (cf. \cite[Theorem 2.3.36]{Arnold:98})
  that the SDE is path-wise unique and admits a unique $C^1$ maximal solution flow $(\phi_t(x), \xi(x))$.
  On the other hand, by Theorem~\ref{thm:rough}, for any $x \in \R^d$, there exists a unique global
  solution $(\psi_t(x,\omega))$ to the RDE
  \begin{equation}
    \dd x_t = b(x_t) \dd t + \sigma(x_t) \dd \mathbf{W}^{\text{It\^o}}_t(\omega), \qquad x_0 = x.
  \end{equation}
  Hence, as $\psi_t(x)$ is a strong solution of the SDE~\eqref{eq:sde-rde} by \cite[Theorem 9.1]{Friz:20},
  we have by uniqueness that there exists a set of full measure $\Omega_x \subseteq \Omega$ such that
  $\psi(x) = \phi(x)$ and $\xi(x) = \infty$ for all $\omega \in \Omega_x$. Moreover, these solutions
  are continuous in $x$ by localization.

  Denoting $\bar \Omega = \cap_{x \in \mathbb{Q}^d} \Omega_x$,
  we have that $\mathbb{P}(\bar \Omega) = 1$ and moreover, for all $x \in \mathbb{Q}^d$ and
  $\omega \in \bar \Omega$, we have that $\psi(x) = \phi(x)$ and $\xi(x) = \infty$. Now,
  as $\phi$ and $\psi{(\cdot)}$ are both continuous in $x$ (where the latter is continuous by
  the continuity of the It\^o-Lyons map), it follows that $\psi(x) = \phi(x)$ and $\xi(x) = \infty$ for all
  $x \in \R^d$ and $\omega \in \bar \Omega$. Consequently, $\phi$ is a $C^1$ global solution
  flow and the SDE~\eqref{eq:sde-rde} is strongly complete.
\end{proof}

We remark that, replacing the driving noise $W$ by a fractional Brownian motion $B^H$ with Hurst parameter
$H \in (\frac{1}{3}, \frac{1}{2}]$, one can interpret the SDE~\eqref{eq:sde-rde} as a rough differential equation.
Thus, by the continuity theorem for rough paths and Theorem~\ref{thm:rough}, we can also conclude
strong completeness for an SDE with multiplicative noise driven by a fractional Brownian motion 
under Assumption~\ref{assume2}.

\begin{appendix}
  \section{Proof of Lemma~\ref{lem:cont-holder}}\label{sec:cont-holder}
  For the convenience of the reader, we provide below a proof for Lemma~\ref{lem:cont-holder}.
  Let $V$ be a Banach space, $A \in \C^{\alpha}(\Delta_T; V)$ be a two parameter process,
  $t_0 \in [0, T)$, and $\alpha' \in (0, \alpha)$. Define
  \begin{align*}
    N : \epsilon\in [0, T - t_0] \mapsto \|A\|_{\alpha'; [t_0, t_0 + \epsilon]}.
  \end{align*}
  
  \begin{lemma}
    Assume $A$ is continuous as a map from $\Delta_T$ to $V$ and vanishes on the diagonal. Then
    $N$ has a point of discontinuity at $\epsilon^*$ if and only if
    \begin{equation}\label{eq:cont-ineq-iff}
      \lim_{h\to 0}\|A\|_{\alpha'; [\epsilon^* - h, \epsilon^* + h]} > N(\epsilon^*).
    \end{equation}
  \end{lemma}
  \begin{proof}
    By translation, we can assume that $t_0 = 0$.   By the monotonicity of the H\"older norm in the interval,
    \eqref{eq:cont-ineq-iff} implies that
    \begin{equation}
      \lim_{h\to 0}N(\epsilon^*+h) \ge  \lim_{h\to 0}\|A\|_{\alpha'; [\epsilon^* - h, \epsilon^* + h]} > N(\epsilon^*)
    \end{equation}
    and the discontinuity of $N$ at $\epsilon^*$. Conversely assuming  $N$ is discontinuous at $\epsilon^*$.
  
    We first note that $N$ is left continuous at any point $\epsilon^*$. Indeed, by continuity of $A$,
    there exist $0\le u^* < v^* \le \epsilon^*$ such that
    $\|A\|_{\alpha'; [0, \epsilon^*]}=\f{\|A_{u^*, v^*}\|}{|u ^*- v^*|^{\alpha'}}$, which means that
    $N(\epsilon)=N(\epsilon^*)$ for any  $\epsilon\in[v^*, \epsilon^*]$. In the case where $v^*<\epsilon^*$,
    $N$ is left continuous at $\epsilon^*$. If on the other hand  $v^*=\epsilon^*$, then
    $$N(\epsilon^*)\ge N(\epsilon^*-h) \ge \frac{\|A_{u^*, v^* - h}\|}{|(v^* - h) - u^*|^{\alpha'}}.$$
    By the continuity of $A$, the right hand side converges to $\|A\|_{\alpha'; [0, \epsilon^*]}=N(\epsilon^*)$
    as $h\to 0$ concluding the left continuity of $N$.
  
    Therefore, $N$ has a discontinuity from the right at $\epsilon^*$. Since $N$ is non-decreasing,
    \begin{equation}\label{eq:cont-ineq}
      \lim_{h \downarrow 0} N(\epsilon^* + h) > N(\epsilon^*).
    \end{equation}
    To show that $ \lim_{h\to 0}\|A\|_{\alpha'; [\epsilon^* - h, \epsilon^* + h]} > N(\epsilon^*)$, it is sufficient if the supremum in the definition of $N(\epsilon^*+h)$ is achieved in  $[\epsilon^* - h, \epsilon^* + h]$ as $h \to 0$.
  
    Let $u^*_h < v^*_h\in [0, \epsilon^* + h]$ be such that $N(\epsilon^* + h)  = \frac{\|A_{u^*_h, v^*_h}\|}{|u^*_h - v^*_h|^{\alpha'}}$.
    Firstly, if $v^*_h < \epsilon^*$, then
    $N(\epsilon^* + h) = N(\epsilon^*)$ contradicting the assumption. Hence  $$v^*_h \in [\epsilon^*, \epsilon^* + h].$$
    If we have $0 \le u^*_h \le \epsilon^* - h$, then
    \begin{equation}\label{eq:cont-ineq-2}
      \begin{split}
        N(\epsilon^*) \le N(\epsilon^* + h) & = \frac{\|A_{u^*_h, v^*_h}\|}{|u^*_h - v^*_h|^{\alpha'}}
        = \frac{\|A_{u^*_h, \epsilon^*} + A_{\epsilon^*, v^*_h} + \delta A_{u^*_h, \epsilon^*, v^*_h}\|}{|u^*_h - v^*_h|^{\alpha'}} \\
                                            & \le \frac{\|A_{u^*_h, \epsilon^*}\|}{|u^*_h - v^*_h|^{\alpha'}}
        + \frac{\|A_{\epsilon^*, v^*_h}\|}{|u^*_h - v^*_h|^{\alpha'}}
        + \frac{\|\delta A_{u^*_h, \epsilon^*, v^*_h}\|}{|u^*_h - v^*_h|^{\alpha'}}                                                 \\
                                            & \le N(\epsilon^*) + \|A\|_{\alpha; [0, T]} h^{\alpha - \alpha'}
        + \frac{\|\delta A_{u^*_h, \epsilon^*, v^*_h}\|}{|u^*_h - v^*_h|^{\alpha'}}.
      \end{split}
    \end{equation}
    If $\liminf_{h \downarrow 0} |u^*_h - v^*_h| = 0$, then $\frac{\|\delta A_{u^*_h, \epsilon^*, v^*_h}\|}{|u^*_h - v^*_h|^{\alpha'}}
      \le 3\|A\|_{\alpha; [0, T]}|u^*_h - v^*_h|^{\alpha - \alpha'}$. Hence, we obtain
    \begin{equation}\label{eq:cont-ineq-3}
      \lim_{h \downarrow 0} N(\epsilon^* + h) = \liminf_{h \downarrow 0} N(\epsilon^* + h) = N(\epsilon^*)
    \end{equation}
    contradicting Equation~\eqref{eq:cont-ineq}.
  
    If, on the other hand, $\liminf_{h \downarrow 0} |u^*_h - v^*_h| >0$,  there exists some $\delta > 0$
    such that for all $h$, $|u^*_h - v^*_h| \ge \delta$, then
    $$\frac{\|\delta A_{u^*_h, \epsilon^*, v^*_h}\|}{|u^*_h - v^*_h|^{\alpha'}}
      \le \delta^{\alpha'} (\|A_{u^*_h, v^*_h} - A_{u^*_h, \epsilon^*}\| + \|A_{\epsilon^*, v_h^*}\|).$$
    Hence, as $v^*_h \in [\epsilon^*, \epsilon^* + h]$,
    the right hand side of which converges to zero as $h \downarrow 0$ by the continuity of $A$.
    Consequently, Equation~\eqref{eq:cont-ineq-3} remains to hold and we again have a contradiction.
  \end{proof}
  
  \begin{proposition}
    Assume $A$ is continuous as a map from $\Delta_T$ to $V$ and vanishes on the diagonal.
    Then $N$ is continuous.
  \end{proposition}
  \begin{proof}
    By the previous lemma, it suffices to show that
    \begin{equation}
      \lim_{h\to 0}\|A\|_{\alpha'; [\epsilon^* - h, \epsilon^* + h]}  \le N(\epsilon^*)
    \end{equation}
    for any $\epsilon^*$. Indeed, this is the case as
    $$\lim_{h\to 0} \|A\|_{\alpha'; [\epsilon^* - h, \epsilon^* + h]}
      \le \|A\|_{\alpha; [0, T]} (2h)^{\alpha - \alpha'} = 0 \le N(\epsilon^*).$$
  \end{proof}
  
  \section{From local solutions to maximal solutions}\label{sec:extension}
  
  Let $(U_N), N \in \N$ be an increasing sequence of bounded open sets with
  with $\bigcup_N U_N = H$. We assume that for all $R > 0$, there exists some $N \in \N$ such that
  $B_R \subseteq U_N$ (note that this assumption is automatically satisfied if $B_R$ is relatively
  compact). Let $b : \R_+ \times H \to H$ be a measurable function which is bounded
  on bounded sets and $b^N : \R_+ \times H \to H$ be such that
  $b^N(t, x) = b(t, x)$ for all $t \in \R_+$ and $x \in U_N$. We in this section outline a
  procedure to construct a maximal solution flow from localized solution flows on $U_N$.
  
  \begin{definition}[Maximal flow]\label{def:flow}
    A maximal flow is a mapping
    $$\phi : D \to H: (s, t, x) \mapsto \Phi_{s, t}(x)$$
    where $D \subseteq \{(s, t) \in \R_+^2 : s \le t\} \times H$ is a non-empty open set such that:
    \begin{enumerate}
      \item For any $x \in H$,  $\{(s, t) : (s, t, x) \in D\}$ contains the diagonal  of $\R_+\times \R_+$.
      \item Setting $\xi(s, x) = \sup \{t : (s, t, x) \in D\}$, $\xi(s, x) > s$ for any $s \in \R_+$ and
            $$(s, \xi(s, x)) \subseteq \{t : (s, t, x) \in D\}.$$
      \item It satisfy the flow properties:
            \begin{itemize}
              \item For each $x \in H$ and $s \in \R_+$, $\phi_{s, s}(x) = x$.
              \item For $x \in H$ and $0 \le s \le r \le t$, we have that $t < \xi(s, x)$ if and only if
                    $r < \xi(s, x)$ and $t < \xi(r, \phi_{s, r}(x))$. If this is the case, then
                    $$\phi_{r, t}(\phi_{s, r}(x)) = \phi_{s, t}(x).$$
            \end{itemize}
    \end{enumerate}
    A maximal flow $\phi$ is said to be global if $\xi = \inf_{x, s} \xi(s, x) = \infty$.
  \end{definition}
  
  Below we build up a maximal solution by localizing the ODE on bounded sets. This construction does not
  require that these localized ODEs satisfy any uniqueness properties.
  \begin{lemma}\label{lem:extend}
    Suppose that for every $N \in \N$ and every initial conditions $(s, x) \in \R_+ \times H$,
    there exists a global solution $t \mapsto \phi^N_{s, t}(x)$ to
    \begin{equation}\label{eq:ode-loc-local}
      \dd x_t = b^N(r, x_t) \dd t + \dd \gamma_t, \qquad x_s=x.
    \end{equation}
    We refer to these as localized solutions. Then, for every initial condition, the ODE
    $$\dd x_t = b(r, x_t) \dd t + \dd \gamma_t$$
  admits a maximal solution $t \mapsto \phi_{s, t}(x)$ in the sense of Definition~\ref{def:sol}. Moreover,
    \begin{itemize}
      \item[(1)]  If the localized solutions are unique, then the maximal solution to \eqref{eq:ode-intro} is also unique.
      \item[(2)]  If, for each $N$, $\phi^N$ is a global flow, then $\phi$ is a maximal flow.
    \end{itemize}\end{lemma}
  
  \begin{proof} 
  Let $s\ge 0$ and
  $t_1 = \inf \{t : \phi_{s, t}^1(x_0) \notin \bar U_1\}$.
Define
    $$\Phi_{s, t}(x_0) = \phi^1_{s, t}(x_0), \qquad t\in [s, t_1].$$
    By construction, $\Phi$ is a solution up to $t_1$. For all $t \le t_1$, we have
    \begin{align*}\Phi_{s, t}(x_0) & = x_0 + \int_s^t b^1(r, \phi^1_{s, r}(x_0)) \dd r + \gamma_t - \gamma_s \\
                                 & = x_0 + \int_s^t b(r, \Phi_{s, r}(x_0)) \dd r + \gamma_t - \gamma_s.
    \end{align*}
Define recursively, for $k\ge 2$,
\begin{align*}
\Phi_{s, t}(x_0) &= \phi^k_{t_{k - 1}, t}(\Phi_{s, t_{k - 1}}(x_0)), \qquad  t\in (t_{k - 1}, t_{k}],\\
      t_k&= \inf \{t \ge t_{k-1}: \phi_{t_{k - 1}, t}^{k}(\Phi_{s, t_{k - 1}}(x_0)) \notin \bar U_k\}.
\end{align*}
Supposing now that $\Phi$ is a solution up to $t_{k-1}$. Then,
    for $t \in (t_{k - 1}, t_k]$,
    \begin{align*}
      \Phi_{s, t}(x_0) & = \phi^k_{t_{k - 1},t}(\Phi_{s, t_{k - 1}}(x_0))                   \\
                       & = \Phi_{t_{k - 1}}(x_0)
      + \int_{t_{k - 1}}^t b^k(r, \phi^k_{t_{k - 1}, r}(\Phi_{s, t_{k - 1}}(x_0))) \dd r
      + \gamma_t - \gamma_{t_{k - 1}}                                                       \\
                       & = x_0 + \int_s^t b(r, \Phi_{s, r}(x)) \dd r + \gamma_t - \gamma_s.
    \end{align*}
   Thus, we have constructed a solution of \eqref{eq:ode-intro} with initial time $s$ up to the random time
    $\xi_s(x_0) = \lim_{k \to \infty} t_k$. By assumption, for all $R > 0$,
    there exists a $k$ such that $B_R \subseteq U_k$. Since $\Phi_{s, t_k}(x_0) \not \in \bar U_k$ by construction,
    it follows that $\lim_{k \to \infty} \|\Phi_{s, t_k}(x_0)\| = \infty$. Consequently,
    $\limsup_{t \uparrow \xi_s(x)}\|\Phi_{s, t}(x_0)\|= \infty$, proving that $(\Phi_{s, t}(x_0), \xi_s(x_0))$
    is a maximal solution to \eqref{eq:ode-intro}.

Now suppose that for each $N \in \N$, \eqref{eq:ode-loc-local} has a unique global solution
    $(\phi^N_{s, \cdot}(x_0), \xi_s^N(x_0))$ with initial time $s$ and initial condition $x_0$.    
          Let $(\Phi_{s, \cdot}(x_0), \xi(s,x_0))$ be a maximal solution to \eqref{eq:ode-intro} and set $\tau^N = \inf\{t : \Phi_{s, t}(x_0) \notin U_N\}$.  Then 
  $$\tilde \phi^N_{s, t}(x_0) = \mathbb{1}_{\{t \le \tau^N\}} \Phi_{s, t}(x_0)
      + \mathbb{1}_{\{t > \tau^N\}} \phi^N_{\tau^N, t}(\Phi_{s, \tau^N}(x_0))$$
      defines a global solution to~\eqref{eq:ode-loc-local} with same initial condition.
     Indeed, for $t < \tau^N$,
    $\tilde\phi^N_{s, t}(x_0)= \Phi_{s, t}(x_0) = x_0 + \int_s^t b^N(r, \tilde \phi^N_r(x_0)) \dd r + \gamma_t - \gamma_r$,
    while for $t > \tau^N$,
    \begin{align*}
      \tilde \phi^N_{s, t}(x_0) & = \phi^N_{\tau^N, t}(\Phi_{s, \tau^N}(x_0))                                \\
                                & = \Phi_{s, \tau^N}(x_0)
      + \int_{\tau^N}^t b^N(r, \phi^N_{\tau^N, r}(\Phi_{s, \tau^N}(x_0))) \dd r + \gamma_t - \gamma_{\tau^N} \\
                                & = x_0 + \int_s^t b^N(r, \tilde \phi^N_{s, r}) \dd r + \gamma_t - \gamma_s.
    \end{align*}
Thus $\tilde \phi_{s, t}^N$ is a global solution to \eqref{eq:ode-loc-local}.
    By uniqueness, $\phi_t^N(x_0)=\tilde \phi_t^N(x_0)$ for all $t$, and consequently,
    \begin{equation}\label{eq:unique-loc}
      \Phi_{s, t}(x_0) = \phi^N_{s, t}(x_0)
    \end{equation}
    on the set $t < \tau^N$, and therefore on $\cup_N \{t : t < \tau^N\} =\{t: t < \xi_s(x_0)\}$,  which concludes the uniqueness of the maximal solution of \eqref{eq:ode-intro}.

    Finally, if the localized SDEs admits a solution flow, then the flow
    property of the maximal solution follows directly from Equation~\eqref{eq:unique-loc}.
  \end{proof}

  The same construction holds for the path-by-path solution for SDEs.
  
  \begin{definition}[Maximal stochastic solution flow]\label{def:path-by-path}
    Let $(W_t)$ be a Brownian motion on $\R^l$. We say the SDE
    \begin{equation}\label{eq:SDE-loc}
      \dd x_t = b(t, x_t) \dd t + \sigma(t, x_t) \dd W_t
    \end{equation}
    has a maximal solution flow $\phi$ if for any $x \in H$, $t \mapsto \phi_{t}(x)$ is
    a strong solution to the SDE~\eqref{eq:SDE-loc} and for almost surely every $\omega$, $\phi(\omega)$
    is a maximal flow.
    \begin{itemize}
      \item If $\phi(\omega)$ is global for almost every $\omega$, we say that $\phi$ is a global
            solution flow and the SDE~\eqref{eq:SDE-loc} has uniform non-explosion.
      \item We say the SDE~\eqref{eq:SDE-loc} is strongly complete if $\phi$ is a global solution flow
            and the map $(t, x) \mapsto \phi_{t}(x, \omega)$ is jointly continuous for almost surely
            every $\omega$.
    \end{itemize}
  \end{definition}

  In the case of additive SDEs, we adopt the same definition of maximal stochastic solution flows, 
  with the Brownian motion replaced by a more general adapted stochastic process. For brevity, we omit 
  the qualifier ``stochastic'' when referring to maximal stochastic solution flows, as the context 
  will typically make clear whether a given flow is stochastic.
  
  
  \begin{corollary}\label{lem:loc-pbp-exist}
    Suppose that, for each $N$, the localized SDE
    \begin{equation}\label{eq:SDE-loc-local}
      \dd x^N_t = b^N(t, x^N_t) \dd t + \dd X_t
    \end{equation}
   admits a path-by-path solution. Then, the SDE~\eqref{eq:SDE-loc} has a maximal path-by-path solution.
 Moreover:
    \begin{itemize}
      \item If the path-by-path solutions of the localized SDEs are unique, then the maximal path-by-path solution of
            the SDE~\eqref{eq:SDE-loc} is also unique.
      \item If the path-by-path solutions of the localized SDEs are adapted, 
            then so is the maximal path-by-path solution of the SDE~\eqref{eq:SDE-loc}.
      \item If the path-by-path solutions of the localized SDEs are unique, are solution flows,
            continuous in $x$ / jointly continuous in $(t, x)$, and the
            constructed maximal path-by-path solution flow for the SDE~\eqref{eq:SDE-loc} has uniform non-explosion,
            then it is also continuous in $x$ / jointly continuous in $(t, x)$.
    \end{itemize}
  \end{corollary}
   
  
  \begin{example}\label{cor:loc-KR}
    Let $b \in L^q([0,T]; L^p_{\loc}(\R^d; \R^d))$ with $p, q \in (2, \infty)$ such that $2/q + d/p< 1$.
    By localizing  $b$ and applying \cite[Proposition 2.1 and Theorem 2.2]{Anzeletti:23}, it
    follows that the SDE $\dd x_t = b(t, x_t) \dd t + \dd W_t$ is strongly complete and 
    path-by-path unique provided uniform non-explosion holds.

In the critical case where $2/q + d/p = 1$ and $d \ge 3$, uniform non-explosion likewise implies strong completeness,
     by the same reasoning together with \cite[Thm. 1.1]{Rockner-Zhao-2021}.
  \end{example}
  
  Another example is when $X = Z$ is a L\'evy process:
  \begin{example}\label{cor:loc-levy}
    Assume now that $b$ is locally Lipschitz in its second argument and the map $t \mapsto \|b(t, 0)\|$
    is bounded on compact sets. Then, defining
    $$b^N(t, x) :=
      \begin{cases}
        b (t, x),                           & \text{if } x \in U_N,    \\
        b\big(t, \frac{Nx}{\|x\|}\big), & \text{if } x \notin U_N,
      \end{cases}$$ 
    $b^N$ is globally Lipschitz. Thus, by \cite[Theorem 3.1 \& 3.2]{Kunita:04}, the SDE
    $$\dd x^N_t = b^N(t, x_t) \dd t + \dd Z_t$$
    has a unique global solution flow $\phi^N$ which is Lipschitz in the initial condition. Thus,
    the SDE $\dd x_t = b(t, x_t) \dd t + \dd Z_t$ has a unique global solution flow which is Lipschitz
    in the initial condition if we have uniform non-explosion.
  \end{example}

\end{appendix}

\newpage
\pagestyle{empty}

\printbibliography
\end{document}